\documentclass[11pt, b5paper]{article}

\makeatletter
\def\blfootnote{\gdef\@thefnmark{}\@footnotetext}
\makeatother

\usepackage{xfrac} 
\usepackage{amsmath,amssymb,amsthm}
\usepackage{mathtools}
\usepackage{enumerate}
\usepackage{dsfont}
\usepackage{mathptmx}
\usepackage[scaled=.90]{helvet}

\usepackage{setspace}

\usepackage{wrapfig}
\usepackage[margin=0.79in]{geometry}
\usepackage{leftidx}
\usepackage{float}
\usepackage[utf8]{inputenc}
\usepackage[hidelinks]{hyperref} 
\usepackage{caption}
\usepackage{import}
\usepackage{hyperref}

\usepackage{tikz, mathrsfs}
\usepackage{tikz-cd}
\usepackage{pgfplots}
\usetikzlibrary{calc}
\usetikzlibrary{intersections, pgfplots.fillbetween}

\usetikzlibrary{arrows}
\newcommand{\midarrow}{\tikz \draw[-triangle 90] (0,0) -- +(.1,0);}

\tikzcdset{scale cd/.style={every label/.append style={scale=#1},
    cells={nodes={scale=#1}}}}

\theoremstyle{definition}
\newtheorem{thm}{Theorem}[section]
\newtheorem{lemma}[thm]{Lemma}
\newtheorem{theorem}[thm]{Theorem}
\newtheorem{corollary}[thm]{Corollary}
\newtheorem{proposition}[thm]{Proposition}
\newtheorem{remark}[thm]{Remark}
\newtheorem{example}[thm]{Example}
\newtheorem{definition}[thm]{Definition}

\def\Map{\mathop{\rm Map}\nolimits}
\def\map{\mathop{\rm map}\nolimits}

\def\op{\mathop{\rm op}\nolimits}
\def\diag{\mathop{\rm diag}}

\def\dirlim{\mbox{\,\hbox{lim}\kern-1.5em
                  \lower1.5ex\hbox{$\longrightarrow$}\,}}

\def\R{{\mathrm{R}}}

\def\F{\mathrm{F}}

\def\C{{\mathcal{C}}}

\def\S{\mathcal{S}}

\def\Set{\text{(Sets)}}

\def\sSet{\text{(Simplicial Sets)}}
\def\ssSet{\text{(Bisimplicial Sets)}}
\def\Ch{\text{(Chain Complexes)}}
\def\dCh{\text{(Double Complexes)}}

\def\Top{\text{(topological spaces)}}

\def\nerve{\mathrm{N}}

\def\reeb{\mathcal{G}}
\def\truncreeb{\mathcal{T}}

\def\homology{\mathrm{H}}
\def\page{\mathrm{E}}

\def\id{\mathrm{id}}

\def\pr{\mathrm{pr}}

\def\Tot{\mathrm{Tot}}

\def\Sh{\mathcal{S}_h}
\def\from{\colon}
\def\to{\rightarrow}

\DeclareMathAlphabet{\mathcal}{OMS}{cmsy}{m}{n}	%keep mathcal font from before mathptmx

\DeclareRobustCommand{\coprod}{\mathop{\text{\fakecoprod}}}
\newcommand{\fakecoprod}{%
  \sbox0{$\prod$}%
  \smash{\raisebox{\dimexpr.9625\depth-\dp0}{\scalebox{1}[-1]{$\prod$}}}%
  \vphantom{$\prod$}%
}	%allows \coprod in mathptmx

\title{Section complexes of simplicial height functions}
\author{Melvin Vaupel, Erik Hermansen and Paul Trygsland }
\date{}

\begin{document}
\blfootnote{\scriptsize{Email: melvin.vaupel@ntnu.no, erik.hermansen@ntnu.no, paul.trygsland@ntnu.no}}

\blfootnote{\scriptsize{Departement of Mathematical Sciences, Norwegian University of Science and Technology, Trondheim, Norway}}
\maketitle

\begin{abstract}
    \noindent A theory of sections of simplicial height functions is developed. At the core of this theory lies the section complex, which is assembled from higher section spaces. The latter encode flow lines along the height, as well as their homotopies, in a combinatorial way. The section complex has an associated spectral sequence, which computes the homology of the height functions domain. We extract Reeb complexes from the spectral sequence. These provide a first order approximation of how homology generators flow along height levels. Our theory in particular models topological section spaces of piecewise linear functions in a completely combinatorial way. 
\end{abstract}

\section{Introduction}
It is a common theme in mathematics to study properties of a space~$X$ through the lens of a real-valued \textit{height function}~$h \from X \to \R$. The best known example of this is probably Morse theory, where~$X$ is a smooth Riemannian manifold and one considers gradient flow lines of a Morse function $h$. There is also a discrete variant of Morse theory due to Forman \cite{forman1998morse}, in which~$X$ is a simplicial complex and~$h$ assigns a real value to every simplex of~$X$.  \\
Given a height function~$h \from X \to \R$ one can also investigate the space of all \textit{sections} of~$h$. A section is a local right inverse to the height function~$h$. If for example~$X$ is a topological space and~$h \from X \to \mathbb{R}$ a continuous function, then a section of~$h$ is a map~$\sigma \from [a,b] \to X$ such that $h \circ \sigma=\id$. In \cite{trygsland2021} it is explained how these sections form the space of morphisms in a topological category- called the \textit{section category} of~$h$. It is then shown that under fairly mild regularity assumptions on~$h$ the classifying space of the section category is weakly equivalent to the base space~$X$. As a consequence the spectral sequence associated to (the nerve of) the section category then computes the homology of~$X$. 
A similar result for smooth Morse functions has been proven in unpublished work by Cohen, Jones and Segal \cite{cohen1995morse}. In their case, the associated spectral sequence reduces to the widely used Morse homology. \\

In this paper, we introduce a theory of sections for the case where~$X$ is a \textit{simplicial set} and $h \from X \to \R$ a simplicial map taking values in a simplicial model of the real line given by the nerve of the poset category~$(\mathbb{R},\leq)$. The fibers of~$h$ assemble into a simplicial set~$\coprod\nolimits_{a\in \mathbb{R}} h^{-1}a$ that we call the~\emph{space of~$0$--sections}. To capture how these fibers are connected across~$p$ levels, we introduce the~$p$--sections. A~$1$--section from height~$a$ to height~$b$ is a~$1$--simplex in~$X$ starting in~$h^{-1}a$ and ending in~$h^{-1}b$, a~$2$--section is a~$2$--simplex labeled by three height values and so on. There is a simplicial set~$(\S_h)_p$, containing the~$p$--sections as its vertices and their~$q$-homotopies as~$q$-simplices. We refer to it as the~\emph{space of~$p$--sections}. It turns out that the section spaces are also simplicial in~$p$. Thus, we may define a bisimplicial set~$\S_h$, associated with the height function~$h$, called the \emph{section complex} of~$h$. The section complex essentially splits up the homology information of~$X$ into two directions: the horizontal direction along~$h$ and the vertical one transversal to it. This is encapsulated in the following main result:

\begin{theorem}
\label{intro:mainresult}
Let~$h\colon X\rightarrow \R$ be a height function. The diagonal of the bisimplicial set~$\S_h$ is homotopy equivalent to~$X$:
\[
    \diag \S_h\simeq X.
\]
\end{theorem}
The computational implications of this result come from the existence of a spectral sequence which computes the homology of the diagonal~$\diag \S_h$, from homological features of the section spaces~$(\S_h)_p$ (see e.g.~\cite{segal1968classifying, goerss2009simplicial}). We refer to the spectral sequence associated to~$\S_h$ as the~\emph{section spectral sequence}. From the first page of the section spectral sequence we extract for~$q=0,1,2,\dots$ chain complexes~$\mathcal{G}_q$ by dividing out degeneracies. We refer to~$\mathcal{G}_q$ as the~$q$th \emph{Reeb complex} of~$h$. One may think of it as an object that carries information about how generators of the~$q$'th homology~$\homology_q$ flow across height levels. Moreover, for a finite simplicial set~$X$, the Reeb complexes can be computed in finite time.
\begin{proposition}
\label{prop:spectralsequenceintro}
Let~$h\colon X\rightarrow \R$ be a height function and~$\mathcal{G}_q$ its~$q$th Reeb complex. The associated section spectral sequence has~$\homology_p \mathcal{G}_q$ appearing as the~$(p,q)$th entry on the second page~$\page^2_{p,q}\simeq \homology_p \mathcal{G}_q$ and converges to the homology of~$X$:
\[
\homology_p \mathcal{G}_q \Rightarrow \homology_{p+q} X.
\]
\end{proposition}
In  \cite{Vaupel_TDA_2022} we define Reeb complexes~$\reeb^A_q$, that capture how generators of homology flow along sections of a \textit{continuous} height function between heights~$A \subset \mathbb{R}$. Based on the theory in \cite{trygsland2021} we can often replace this object with a \textit{truncated Reeb complex}~$\truncreeb_q$ that only incorporates sections between \textit{adjacent critical levels}. The reason is that continuous sections (and their homotopies) always factorise into smaller sections between intermediate height levels. As a consequence the continuous section spectral sequence always converges on the second page. This is different from the simplicial setting of the paper at hand, where the presence of higher section spaces can lead to non-trivial higher differentials in the spectral sequence.
 We say that~$X$ is \emph{subdivided according to}~$h$ if every~$1$--section traverses only successive height levels. In this case, the section spectral sequence collapses at the second page. We are then in a position to directly compare the topological and the simplicial Reeb complexes. 
 \begin{proposition}\label{proposition:comparereeb}
 Let~$h \from X \to \R$ be a simplicial height function, such that~$X$ is subdivided according to~$h$. Then~$h$ associates to a piecewise linear function~$f \from |X| \to \mathbb{R}$ and there is an induced isomorphism of chain complexes between~$\reeb_h$ and the truncated continuous Reeb complex~$\truncreeb_q$.
 \end{proposition}

This result shows in particular that simplicial section spaces can model the homology theory of topological sections of piecewise linear functions on CW-complexes.  
 
Because the topological Reeb complexes always compute the homology of the base space it also follows that in the case where~$X$ is subdivided according to~$h$ the  Reeb
 complexes~$\mathcal{G}_q$ recover the homology of~$X$ directly. For field coefficients we then have:
\begin{equation*}
    \homology_{n}X=\bigoplus_{p+q=n}\homology_p\mathcal{G}_q.
\end{equation*}

\textbf{Outline.} In 
Section~\ref{section:SectionComplex} we start out by briefly reviewing some basic concepts from the theory of simplicial sets. This sets the stage to then  define section spaces and describe how they assemble into the \textit{section complex}. We proof Theorem~\ref{intro:mainresult} in Section~\ref{section:mainresult}. Subsequently Reeb complexes are introduced in  Section~\ref{section:ReebGraphs}. This is supplemented by a discussion of the more complicated section spectral sequence from which the Reeb complexes can be extracted. In section \ref{section:ss} we briefly provide some general background on spectral sequences of double complexes. Throughout section \ref{section:ReebComplexes} numerous examples are computed to illustrate the theory. Finally Section~\ref{section:topologicalspaces!?} relates our simplicial theory back to the theory of topological sections in~\cite{trygsland2021} by providing a proof of Proposition \ref{proposition:comparereeb}. The paper closes with some final remarks on the usefulness of simplicial sets in computational topology.  

\textbf{Computer code.} A Python implementation for computing section complexes, as well as Reeb complexes, is found at \url{https://github.com/paultrygs/Section-Complex/}.%\cite{Vaupel_SectionComplex_2021}

\textbf{Notation.} Categories of familiar objects are put inside parenthesis, e.g.~$\sSet$. The hom-set of maps~$X\rightarrow Y$ is denoted~$\Map(X,Y)$, while~$\map(X,Y)$ refers to the simplicial set of maps. We will in general consider chain complexes over coefficients in some field~$k$.

\newpage

\section{The section complex}
\label{section:SectionComplex}
We introduce the central object of our studies: the \textit{section complex}~$\S_h$ associated to a simplicial height function~$h \from X \to \R$. We assemble it from the \textit{spaces of sections}~$(\S_h)_p$. As such it contains the topological information about sections between height levels of~$h$ but also the combinatorial information about how these fit together across different height levels. Theorem \ref{intro:mainresult}, that we proof at the end of this section, shows that we can combine these two pieces of information to recover the homology of~$X$.  \\

Let us start out with a quick review of some basics about \textit{simplicial sets}. This is not intended to be very pedagogical but rather serves the purpose to quickly introduce some necessary notation. For a comprehensive introduction to the theory of simplicial sets see for example \cite{goerss2009simplicial} and for a more elementary treatment \cite{friedman2021elementary}.   

\subsection{Background on simplicial sets}
\label{section:SimplicialSets}
A~\emph{simplicial set}~$X$ is a sequence~$X_n$ of sets, ranging over~$n=0,1,2,\dots$, together with face maps~$d_i\colon X_n\rightarrow X_{n-1}$ and degeneracy maps~$s_j\colon X_n\rightarrow X_{n+1}$ satisfying certain relations~\cite[p.4]{goerss2009simplicial}. An element~$x$ in~$X_n$ is interpreted as an~$n$--simplex whose~$i$th face is~$d_ix$, whereas~$s_jx$ incorporates ways to consider~$x$ as an~$(n+1)$--simplex. In contrast to simplicial complexes, this for example implies that an~$(n+1)$--simplex~$y$ can have an~$(n-1)$--simplex~$x$ as its face;~$d_i y=s_j x$. Moreover, two distinct~$n$--simplices~$x$ and~$y$ can have equal faces~$d_ix=d_i y$ for all~$i$. Equivalently, the data of a simplicial set~$X$ can be organized into a functor~$X\colon \Delta^{\op}\rightarrow \Set$, where~$\Delta$ is the~\emph{simplex category}.
\begin{example}
We construct a circle from~$0$--simplices~$v_0$ and~$v_1$ and~$1$--simplices~$e_0$ and~$e_1$, not counting degeneracies, by declaring~$d_0e_i=v_1$ and~$d_1e_i=v_0$, for $i=0,1$. A sphere can be obtained from a single~$0$--simplex~$v$ and~$2$--simplex~$f$. In this case, all faces of~$f$ must be equal to~$s_0v$, a degenerate~$0$--simplex. This means that the boundary of~$f$ is equal to the point~$v$.

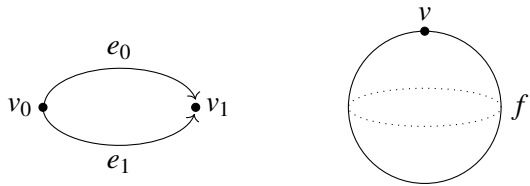
\begin{center}
    \begin{tikzpicture}
    
      \filldraw (0,0) circle (1.5pt);
      \filldraw (2,0) circle (1.5pt);
      \draw[->] (0,0) arc [start angle=180,end angle=350, x radius=1cm, y radius=0.5cm];
      \draw[<-] (2,0.1) arc [start angle=10,end angle=180, x radius=1cm, y radius=0.5cm];
      \node[left] at (0,0) {$v_0$};
      \node[right] at (2,0) {$v_1$};
      
      \node[above] at (1,0.5) {$e_0$};
      \node[below] at (1,-0.5) {$e_1$};

      \filldraw (5,1) circle (1.5pt);
      \draw[] (5,1) arc [start angle=90,end angle=450, x radius=1cm, y radius=1cm];
      \draw[dotted] (6,0) arc [start angle=0,end angle=360, x radius=1cm, y radius=0.25cm];
      \node[above] at (5,1) {$v$};
      \node[right] at (6,0) {$f$};
      
    \end{tikzpicture}
\end{center}
\end{example}
A~\emph{simplicial map}~$f\colon X\rightarrow Y$ is a series of maps~$f_n\colon X_n\rightarrow Y_n$ which commutes with face and degeneracy maps. 
%Such data is equivalent to a map between presheaves~$\Delta^{\op}\rightarrow \Set$. 
Pictures as above are produced by labeling each~$n$--simplex in~$X$ with the topological~$n$--simplex~$\Delta^n_t$ and identifying appropriate simplices via the~\emph{geometric realization}~$|X|=(\coprod\limits_n X_n \times \Delta^n_t )/ \sim$. The quotient glues simplices along faces and collapses degenerate simplices. Note that the realization defines a functor
\[
|\cdot| \from \sSet \to \Top
\]

Any small category~$\C$ defines a simplicial set~$\nerve \C$, called the~\emph{nerve} of~$\C$. The set of~$0$--simplices,~$\nerve\C_0$, consists of objects in~$\C$, and~$\mathrm{N}\mathcal{C}_n$ consists of tuples~$(m_1,\dots,m_n)$, of composable morphisms within~$\mathcal{C}$. The~$i$th face of~$(m_1,\dots,m_n)$ is determined by composition~$(m_1,\dots,m_{i+1}\circ m_i,\dots,m_n)$ for~$i\neq 0,n$, whereas~$d_0$ and~$d_n$ drops~$m_0$ and~$m_n$, respectively. We depict a~$2$--simplex~$(f,g)$:
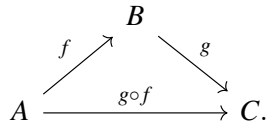
\begin{center}
    \begin{tikzcd}
    & B \arrow[rd,"g"] &  \\
    A \arrow[ru,"f"] \arrow[rr,"g\circ f"] & &  C.
\end{tikzcd}
\end{center} 
\begin{example}
\label{example:[n]}
Let~$[n]$ be the category generated by the directed graph~$0\rightarrow 1\rightarrow\cdots\rightarrow n$. Applying the nerve yields the \emph{standard simplicial}~$n$--simplex~$\Delta^n=\nerve[n]$. It consists of a unique~$n$--simplex coming from the tuple~$(0\rightarrow 1,1\rightarrow 2,\dots,n-1\rightarrow n)$ with~$n+1$ distinct faces. We recover the topological~$n$--simplex as~$|\Delta^n|$.
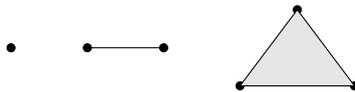
\begin{center}
    \begin{tikzpicture}
    \filldraw (0,0) circle (1.5pt);
    
    \filldraw (1,0) circle (1.5pt);
    \filldraw (2,0) circle (1.5pt);
    \draw (1,0) -- (2,0);
    
    \filldraw (3,-0.5) circle (1.5pt);
    \filldraw (4.5,-0.5) circle (1.5pt);
    \filldraw (3.75,0.5) circle (1.5pt);
    \draw[fill=gray!20] (3,-0.5) -- (4.5,-0.5) -- (3.75,0.5) -- cycle;
    \end{tikzpicture}
\end{center}
There are \emph{simplicial inclusions}~$\delta_i\colon \Delta^n \rightarrow \Delta^{n+1}$ which identify~$\Delta^n$ with the~$i$th face of~$\Delta^{n+1}$. Observe that~$\delta^i(q)$ equals~$q$ if~$q<i$ and~$q+1$ otherwise. Conversely, there are \emph{simplicial collapses}~$\sigma^j\colon \Delta^{n+1}\rightarrow \Delta^{n}$ for which~$\sigma^i(q)$ equals~$q$ if~$n\leq i$ and~$q-1$ otherwise. 
\end{example}
\begin{definition}
\label{def:realline}
Let~$(\mathbb{R},\leq)$ be the real line equipped with its usual ordering. We define the~\emph{simplicial real line}~$\R=\nerve (\mathbb{R},\leq)$.
\end{definition}
An~$n$--simplex in~$\R$ is uniquely determined by a non-decreasing sequence~$\bar{a}=(a_0,\dots,a_n)$ of real numbers. 

Two simplicial sets~$X$ and~$Y$ define a product~$X\times Y$ with~$n$--simplices~$X_n\times Y_n$, whose face and degeneracy maps are computed component-wise.
\begin{example}
\label{example:standard1timesstandard1}
Consider the product of two copies of the standard~$1$--simplex:~$\Delta^1\times \Delta^1$. Decomposing~$0\rightarrow 1$ in components~$(0\rightarrow 1,0\rightarrow 1)=(0\rightarrow 1,1\rightarrow 1)\circ (0\rightarrow 0, 0 \rightarrow 1)$ yields the top~$2$--simplex in its realization:
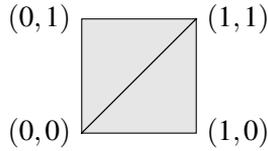
\begin{center}
    \begin{tikzpicture}[scale=1.5]
    \node[left] at (0,0) {$(0,0)$};
    \node[left] at (0,1) {$(0,1)$};
    \node[right] at (1,0) {$(1,0)$};
    \node[right] at (1,1) {$(1,1)$};
    \draw[fill=gray!20] (0,0) -- (1,1) -- (0,1) -- cycle;
    \draw[fill=gray!20] (0,0) -- (1,1) -- (1,0) -- cycle;
    \end{tikzpicture}
\end{center}
The bottom one is obtained as~$(0\rightarrow 1,0\rightarrow 1)=(1\rightarrow 1,0\rightarrow 1)\circ (0\rightarrow 1, 0 \rightarrow 0)$.
\end{example}
A \emph{simplicial homotopy} is a simplicial map~$H\colon X\times I\rightarrow Y$ such that~$I$ realizes to the standard unit interval. Note that a simplicial homotopy realizes to an ordinary homotopy in topological spaces~\cite{segal1968classifying}.

The~\emph{simplicial mapping space}~$\map(X,Y)$ contain simplicial maps~$f\colon X\rightarrow Y$ as its~$0$--simplices. An~$n$--simplex in $\map(X,Y)_n$ is a simplicial map~$X\times \Delta^n \rightarrow Y$. Since~$\Delta^1$ is a model of the interval, the~$1$--simplices are homotopies. Face and degeneracy maps,~$d_i$ and~$s_j$, are obtained by pre-composing with component-wise maps~$\id_X\times \delta^i$ and~$\id_X\times \sigma^j$, respectively (Example~\ref{example:[n]}). Applying~$d_i$ to~$f\colon X\times \Delta^n\rightarrow Y$ thus restricts~$\Delta^n$ to its~$i$th face, whereas~$s_j$ adds appropriate identities.

\begin{example}
\label{example:mappingspace}
A~$0$--simplex in~$\map(\Delta^1,X)$ is a simplicial map~$e\colon \Delta^1\rightarrow X$, uniquely determined by a~$1$--simplex~$e$ in~$X$. Homotopies~$H\colon \Delta^1\times \Delta^1\rightarrow X$, or~$1$--simplices, are determined by squares in~$X$ connecting two~$1$--simplices~$e_0$ and~$e_1$.
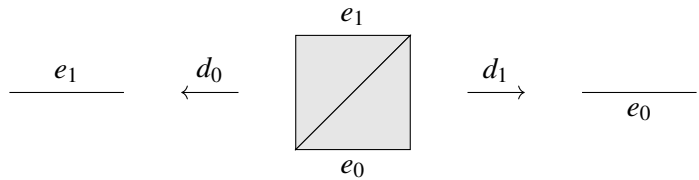
\begin{center}
    \begin{tikzpicture}[scale=1.5]
    \draw[fill=gray!20] (0,0) -- (1,1) -- (0,1) -- cycle;
    \draw[fill=gray!20] (0,0) -- (1,1) -- (1,0) -- cycle;
    \node[above] at (0.5,1) {$e_1$};
    \node[below] at (0.5,0) {$e_0$};
    
    \draw[->] (-0.5,0.5) -- (-1,0.5);
    \draw[->] (1.5,0.5) -- (2, 0.5);
    \node[above] at (-0.75,0.5) {$d_0$};
    \node[above] at (1.75,0.5) {$d_1$};
    
    \draw (-2.5,0.5) -- (-1.5,0.5);
    \draw (2.5,0.5) -- (3.5,0.5);
    \node[below] at (3,0.5) {$e_0$};
    \node[above] at (-2,0.5) {$e_1$};
    \end{tikzpicture}
\end{center}
\end{example}

\subsection{Sections of height functions}\label{section:sectionsofheightfunction}
We are now going to use the theory of simplicial sets from above to model the behaviour of sections of a height function. Before we do so lets quickly remind ourselves of the \textit{continuous} theory of sections from \cite{trygsland2021}.
%Let~$f \from T \to \mathbb{R}$ be a continuous function on a topological space.
%In ~\cite{trygsland2021}, a~\emph{section} of~$f$ is defined as a map~$\rho \colon [a,b]\rightarrow T$, such that the composition~$f\circ \rho$ is the inclusion~$[a,b]\hookrightarrow \mathbb{R}$. Denote by~$\map([a,b],T)$ the topological space of maps~$[a,b]\rightarrow T$, with the compact-open topology, and define~$\mathrm{Sect}_f[a,b]$ as the subspace whose points are the sections of~$f$. These are then arranged in the space of all sections~$\mathrm{Sect}_f=\coprod \mathrm{Sect}_f[a,b]$, ranging over all real numbers~$a\leq b$. Note that two sections in~$\mathrm{Sect}_f[a,b]$ and~$\mathrm{Sect}_f[b,c]$, with compatible ending and starting points, can be concatenated to a section in~$\mathrm{Sect}_f[a,c]$. This makes it possible to define the \textit{section category} of~$f$, a category internal to topological spaces with~$\mathrm{Sect}_f$ as its space of morphisms. In~\cite{trygsland2021} it is then shown that under fairly mild assumptions, the classifying space of this category is homotopy equivalent to~$T$. These assumptions are for example met by piecewise linear functions. \\

Let~$f \from T \to \mathbb{R}$ be a continuous function on a topological space.
In ~\cite{trygsland2021}, a~\emph{section} of~$f$ is defined as a map~$\rho \colon [a,b]\rightarrow T$, such that the composition~$f\circ \rho$ is the inclusion~$[a,b]\hookrightarrow \mathbb{R}$. Denote by~$\map([a,b],T)$ the topological space of maps~$[a,b]\rightarrow T$, with the compact-open topology, and define~$\mathrm{Sect}_f[a,b]$ as the subspace whose points are the sections of~$f$. These are then arranged in the space of all sections~$\mathrm{Sect}_f=\coprod \mathrm{Sect}_f[a,b]$, ranging over all real numbers~$a\leq b$. Note that two sections in~$\mathrm{Sect}_f[a,b]$ and~$\mathrm{Sect}_f[b,c]$, with compatible ending and starting points, can be concatenated to a section in~$\mathrm{Sect}_f[a,c]$. This makes it possible to define the \textit{section category} of~$f$, a category internal to topological spaces with~$\mathrm{Sect}_f$ as its space of morphisms. In~\cite{trygsland2021} it is then shown that under fairly mild assumptions, the classifying space of this category is homotopy equivalent to~$T$. These assumptions are for example met by piecewise linear functions. \\
We will now describe how to obtain such piecewise linear functions from height functions on  simplicial sets.
To that end, recall the definition of the simplicial real line~$\R$ (Definition~\ref{def:realline}).
\begin{definition}
\label{definition:heightfunction}
Let~$X$ be a simplicial set. A height function~$h$ on~$X$ is defined as a simplicial map~$h \from X \to \R$.
\end{definition}
We can equivalently characterize a height function,~$h$, as a map between sets~$h\colon X_0\rightarrow \mathbb{R}$. Indeed,~$h$ associates to each~$0$--simplex~$v$ in~$X_0$ a height~$h(v)$ in~$\mathbb{R}$. Conversely, assigning to every~$v $ in~$ X_0$ a height $h(v)$ such that the orientation of the~$1$--simplices in~$X$ is respected, defines a unique height function $h \from X \to \R$. We call the image $h(X_0)$ the \textit{height levels} of~$h$.   

%In the previous section it was explained that we can realize a simplicial set~$X$ as a topological space~$|X|$. 
We observe from the previous section that a point in the realization~$|\R|$ is a class~$[\bar{a},\bar{t}]$ where~$\bar{a}=(a_0,\dots,a_n)$ is a non-decreasing sequence of real numbers and~$\bar{t}=(t_0,\dots,t_n)$ a point in the topological~$n$--simplex. The dot product~$\bar{a}\bar{t}$ defines a continuous function~$c\colon |\R|\rightarrow \mathbb{R}$ from the realization of~$\R$ to the real line. Any height function~$h\colon X\rightarrow \R$ thus associates to a piecewise linear function~$f\colon |X|\rightarrow \mathbb{R}$ by composing~$|h|$ and~$c$. It is shown in~\cite{trygsland2021} that no homotopical information is lost if we only consider those sections that start and end at the height levels of~$h$. 
 
We now ask the following question: is it possible to construct a \textit{simplicial version} of the section category directly from the simplicial height function~$h \from X \to \R$ rather than from the associated piecewise linear function~$f$? This would render all the information involved combinatorial and in particular accessible to computational topology. 

There is a natural choice for replacing the topological space of sections between two heights~$a_0$ and~$a_1$ with a simplicial set, $\Sh[a_0,a_1]$. Namely, the subspace of the mapping space,~$\map(\Delta^1,X)$, carved out by the pullback:
\begin{equation*}
		\begin{tikzcd}
		 \Sh[a_0,a_1] \arrow[r] \arrow[d] & \map(\Delta^1,X) \arrow[d,"{\map(\Delta^1 ,h)}"] \\
	  \Delta^0 \arrow[r, "{(a_0,a_1)}"] & \map(\Delta^1,\R),
		\end{tikzcd}
	\end{equation*}
where we interpret the~$1$--simplex~$(a_0,a_1)$ in~$\R$ as the simplicial map~$\Delta^1 \to \R$, with~$0\mapsto a_0$ and~$1\mapsto a_1$. We may then define~$(\Sh)_1$, \textit{the space of~$1$--sections} of~$h$ as the disjoint union 
\begin{equation*}
    (\Sh)_1 = \coprod_{a_0\leq a_1} \Sh[a_0,a_1].
\end{equation*}

\begin{example}
    \label{example:standard1}
	We take as our simplicial set the standard~$2$--simplex~$\Delta^2$ and define a height function by the labels of the figure:
	
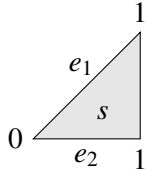
\begin{center}
		\begin{tikzpicture}[scale= 0.7]
		
			\node[left] at (0,0) {$0$};
			\node[above] at (2,2) {$1$};
			\node[below] at (2,0) {$1$};
			
			\draw[fill=gray!20] (0,0) -- (2,2) -- (2,0) -- cycle;
			
			\node[below] at (1,0) {$e_2$};
			\node[above] at (0.9,1) {$e_1$};
			\node[right] at (1,0.5) {$s$};
		
		\end{tikzpicture}
	\end{center}
	The two horizontal~$1$--simplices~$e_1$ and~$e_2$ in~$\Delta^2$ are~$0$--simplices in the section space~$\S_h[0,1]$. The~$2$--simplex~$s$ corresponds to a~$1$--simplex connecting~$e_1$ and~$e_2$;~$d_0s=e_1$ and~$d_1s=e_2$.   
\end{example}
However, the following example illustrates why we cannot proceed as in the construction of the topological section category.  
%\begin{example}\label{example:21horn}
%Take the~$(2,1)$-horn~$\Lambda^2_1$. 
%\begin{center}
%\begin{tikzpicture}[scale= 0.7]

 %   \node[left] at (0,0) {$0$};
  %  \node[above] at (2,2) {$1$};
   % \node[right] at (4,0) {$2$};
    
    %\draw (0,0)-- (2,2);
    %\draw (2,2)-- (4,0);
    
%\end{tikzpicture}
%\end{center}
%Define on it a height function according to the labeling of vertices  as given in the picture. There is a unique~$1$-simplex between heights~$0$ and~$1$ and another one between heights~$1$ and~$2$- giving~$\Sh[0,1]=\{\ast\}$ and~$\Sh[1,2]=\{\ast\}$. These two sections fit together at height~$1$ and working in the realization we would indeed be able to concatenate them. As there is no~$1$-simplex going between height~$0$ and height~$2$ we can't model this behavior with the simplicial section space. There we loose the possibility to compose by concatenation.  
%\end{example}
%It seems like an easy fix for the problem pointed out in this example could be to formally add compositions to~$(\Sh)_1$ in order to obtain a simplicial category. 
%The following example illustrates however that the solution can not be that easy. 
\begin{example}
\label{example:standard2} We define a height function on the standard~$2$--simplex~$\Delta^2$ as follows:
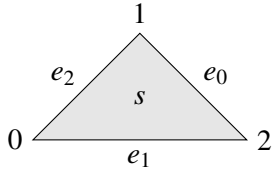
\begin{center}
\begin{tikzpicture}[scale= 0.7]

    \node[left] at (0,0) {$0$};
    \node[above] at (2,2) {$1$};
    \node[right] at (4,0) {$2$};
    
    \draw[fill=gray!20] (0,0) -- (2,2) -- (4,0) -- cycle;
    
    \node[below] at (2,0) {$e_1$};
    \node[left] at (1,1.2) {$e_2$};
    \node[right] at (3,1.2) {$e_0$};
    \node[] at (2,0.8) {$s$};
    
\end{tikzpicture}
\end{center}
 There is a unique~$1$--simplex between each distinct pair of heights. This means that the~$1$--section spaces are:~$\Sh[0,1]=\{e_2\}$,~$\Sh[1,2]=\{e_0\}$ and~$\Sh[0,2]=\{e_1\}$. Hence the space of $1$--sections cannot be utilized to recover the~$2$--simplex~$s$ in~$\Delta^2$ connecting~$e_0$,~$e_1$ and~$e_2$. Contrast this with the corresponding topological situation. In that case, the sections~$|e_2|$ and~$|e_0|$ may be concatenated into~$|e_2|\ast |e_0|$, a section from height~$0$ to height~$2$ running through~$0\rightarrow 1\rightarrow 2$. We would also find a continuous path from~$|e_2|\ast |e_0|$ to~$|e_1|$ moving along the bottom of the triangle, given by continuous deformation through the realization of the~$2$--simplex. In particular, this process encodes the topological information provided by~$|s|$. 
\end{example} 

%Given an arbitrary simplicial height function~$h\from X \to \R$ we could potentially overcome these problems by finding a clever replacement~$h' \from X' \to \R$ where~$X'$ is subdivided according to the height function~$h'$.
%We can then expect all the homotopical information to be contained in the spaces~$\Sh[a_0,a_1]$ for adjacent heights~$a_0$ and~$a_1$. This is actually true. In section \ref{section:topologicalspaces!?} we will make the notion of \textit{subdivided according to a height function} formal and then prove an appropriate result. \\ 
%In order to obtain a simplicial section category we would have to formally add compositions to~$(\Sh)_1$. While this is an easy task in Example \ref{example:21horn} the combinatorics of adding formal compositions of~$n$-simplices in~$(\Sh)_1$ get very complicated for~$n>1$. This is reminiscent of the complications in describing compositions in higher category theory. \\           
%We will thus take another route here. Instead of trying to create the appropriate space of morphisms for a category of sections and then turn it into a simplicial space by applying a nerve functor we will define the spaces of~$n$-simplices directly. This leads us to the concept of \textit{higher section spaces}.  
To recover higher simplices across more than two height levels, we introduce~\emph{higher sections}.

\subsection{Higher sections and the section complex}

\begin{definition}
\label{definition:ksections}
	Given a height function~$h\from X\to \R$ on a simplicial set~$X$, we may construct for every~$p$--simplex~$\bar{a} \from \Delta^p \to \R$, a simplicial set~$\Sh[\bar{a}]$ as the pullback  
 \begin{equation*}
      \begin{tikzcd}
       \Sh[\bar{a}] \arrow[r] \arrow[d] & \map(\Delta^p,X) \arrow[d,"{\map(\Delta^p,h)}"] \\
    \Delta^0 \arrow[r, "\bar{a}"'] & \map(\Delta^p,R).
      \end{tikzcd}
  \end{equation*}
Taking the disjoint union over all~$p$--simplices in~$\R$, we obtain 
\begin{equation*}  
      (\Sh)_p := \bigsqcup_{\bar{a} \in \R_p} \Sh[\bar{a}],
\end{equation*}
which we refer to as the \textit{space of~$p$--sections of~$h$}. 
\end{definition}
By definition, the space of~$p$--sections is a simplicial set. This is guaranteed by the fact that pullbacks and coproducts always exist in the category of simplicial sets. The~$0$--simplices are the \textit{$p$--sections of~$h$}, i.e. the ~$p$--simplices in~$X$ spanning~$p$ height levels. The ~$q$--simplices are then the~$(p,q)$--sections in~$(\S_h)_p$ corresponding to a simplicial map~$\rho \from \Delta^p \times \Delta^q \to X$ such that there is some~$p$--simplex~$\bar{a}=(a_0,\dots, a_p)$ in~$ R_p$ for which 
\begin{equation*}\label{explicitsections}
    \begin{tikzcd}
        \Delta^p \times \Delta^q \arrow[d,"\rho"] \arrow[r,"\pr_0"] & \Delta^p \arrow[r,"\bar{a}"] & \R \\
        X \arrow[urr,"h"']
    \end{tikzcd}
\end{equation*}
commutes. For instance, the 2-simplex, $s$, of Example \ref{example:standard1} is  a $(1,1)$--section in $(S_h)_1$. In Example~\ref{example:standard2}, on the other hand,~$s$ is a~$(2,0)$--section in~$(S_h)_2$.

In Section \ref{section:SimplicialSets}, we explained how the face and degeneracy maps in~$\map(\Delta^p,X)$ are obtained by pre-composing with~$\id \times  \delta^i$ and~$\id \times \sigma^j$:
\begin{center}\label{explicitsections}
    \begin{tikzcd}
        \Delta^p \times \Delta^{q-1} \arrow[dd,bend right=70,"d^v_i(\rho)"'] \arrow[d,"\id \times  \delta^i"] \arrow[r,"\pr_0"] & \Delta^p \arrow[r,"\bar{a}"] & \R &
        \Delta^p \times \Delta^{q+1} \arrow[dd,bend right=70,"s^v_j(\rho)"'] \arrow[d,"\id \times  \sigma^j"] \arrow[r,"\pr_0"] & \Delta^p \arrow[r,"\bar{a}"] & \R 
        \\
        \Delta^p \times \Delta^{q} \arrow[d,"\rho"] & & & \Delta^p \times \Delta^{q} \arrow[d,"\rho"] & & \\
        X \arrow[uurr,"h"'] & & & X \arrow[uurr,"h"'] & & 
    \end{tikzcd}
\end{center}
These diagrams characterize the face and degeneracy maps
\[
d_i^v \from  (\Sh)_{p,q} \to (\Sh)_{p,q-1} \quad \text{and} \quad s_j^v \from  (\Sh)_{p,q} \to (\Sh)_{p,q+1}
\]
in the space of~$p$--sections. As the superscript~$v$ indicates, we call these the \textit{vertical face and degeneracy maps}. 
%This concludes the explicit description of the spaces of~$p$--sections. 

The above use of `vertical' hints to the fact that there is a second, \textit{horizontal}, simplicial structure. Indeed, we can pre-compose a~$(p,q)$--section with a simplicial inclusion or collapse applied to the first component in~$\Delta^p \times \Delta^q$ to obtain commutative diagrams   
\begin{equation*}
\begin{tikzcd}
    \Delta^{p-1} \times \Delta^q \arrow[dd,bend right=70,"d^h_i(\rho)"'] \arrow [d,"\delta^i \times \id"] \arrow[r,"\pr_0"] & \Delta^{p-1} \arrow[r,"\delta^i"] & \Delta^p \arrow[r,"\bar{a}"] & \R \\
    \Delta^{p} \times \Delta^q \arrow [d,"\rho"] & & & \\
    X \arrow[uurrr,"h"'] & & &
    \end{tikzcd}
\end{equation*}
and 
\begin{equation*}
\begin{tikzcd}
    \Delta^{p+1} \times \Delta^q \arrow[dd,bend right=70,"s^h_j(\rho)"'] \arrow [d,"\sigma^j \times \id"] \arrow[r,"\pr_0"] & \Delta^{p+1} \arrow[r,"\sigma^j"] & \Delta^p \arrow[r,"\bar{a}"] & \R \\
    \Delta^{p} \times \Delta^q \arrow [d,"\rho"] & & & \\
    X \arrow[uurrr,"h"'] & & &
    \end{tikzcd}
\end{equation*}
These characterize maps of sets \begin{equation*}
    d^h_i \from (\Sh)_{p,q} \to (\Sh)_{p-1,q} \quad \text{and} \quad 
    s^h_j \from (\Sh)_{p,q} \to (\Sh)_{p+1,q}
\end{equation*}
which we refer to as the \textit{horizontal face and degeneracy maps}. Alternatively, the horizontal face maps can be induced from the universal property of the pullback via:
\begin{center}
\begin{tikzcd}
       & & \map(\Delta^p,X) \arrow[dd] \arrow[dr,"{\map(\delta^i,X)}"] & \\ & & & \map(\Delta^{p-1},X) \arrow[dd] \\
    \Delta^0 \arrow[rr, "\bar{a}"] \arrow[dr,equal] & & \map(\Delta^p,\R) \arrow[dr,"{\map(\delta^i ,\R)}"] & \\
    & \Delta^0 \arrow[rr,"\bar{a}\circ \delta^i"] & & \map(\Delta^{p-1},\R), 
    \end{tikzcd}
\end{center}
and similarly for horizontal degeneracy maps. This shows that the horizontal face maps are in fact simplicial maps 
\begin{equation*}
    d^h_i \from (\Sh)_p \to (\Sh)_{p-1} \quad \text{and} \quad  s^h_j \from (\Sh)_p \to (\Sh)_{p+1}
\end{equation*}
going from~$p$--sections to~$(p-1)$ and~$(p+1)$--sections, respectively. The intuition is that~$d_i^h$ restricts a~$p$--section in~$\S_h[a_0,\dots,a_p]$ to a~$(p-1)$--section in~$\S_h[a_0,\dots,\hat{a}_i,\dots,a_p]$, whereas~$s_j^h$ adds a degenerate label~$\S_h[a_0,\dots,a_j,a_j,\dots,a_p]$. The set~$(\S_h)_{p,q}$ is therefore simplicial in both~$p$ and~$q$, defining a~\emph{bisimplicial set}. 
\begin{definition}\label{defintion:sectioncomplex}
	The \textit{section complex} of a height function~$h \from X \to \R$ is the bisimplicial set~$\Sh$ with~$(p,q)$--simplices given by~$(\Sh)_{p,q}$, i.e. the~$(p,q)$--sections~$\rho\colon \Delta^p\times \Delta^q\rightarrow X$. It has horizontal and vertical face and degeneracy maps as defined above. 
\end{definition}

\begin{remark}\label{remark:bisimplicialsets}
We didn't explain what a bisimplicial set in general is.
Similar to a simplicial set it is given by a sequence~$X_{p,q}$ of sets that now ranges over pairs of natural numbers. The first component of these indices is sometimes called the \textit{horizontal} direction and the second one is called the \textit{vertical} direction. Correspondingly there are also two types of face- and degeneracy-maps for bisimplicial sets. The horizontal ones:
\begin{equation*}
    d^h_i \from X_{p,q} \to X_{p-1,q} \quad \text{and} \quad 
    s^h_j \from X_{p,q} \to X_{p+1,q}
\end{equation*}
and the vertical ones:
\[
d_i^v \from  X_{p,q} \to X_{p,q-1} \quad \text{and} \quad s_j^v \from  X_{p,q} \to X_{p,q+1}.
\]
These are both required to satisfy the \textit{simplicial relations}. We can also define a bisimplicial set as a functor
\begin{equation*}
    \Delta^{\op} \times \Delta^{\op} \to \Set
\end{equation*}
where $\Delta$ denotes the simplex category. The reader familiar with adjunctions will agree that this information is then equivalently presented as a functor 
\begin{equation*}
    \Delta^{\op} \to \left( \Delta^{\op} \to \Set \right)
\end{equation*}
 Postcomposing with the realization functor exhibits the intimate relationship of bisimplicial sets and simplicial spaces. In particular we may turn the section complex~$\S_h$ from Definition \ref{definition:sectioncomplex} into a simplicial space~$\mathrm{T} \S_h$: 
 \begin{equation*}
 \begin{tikzcd}
     \Delta^{\op} \arrow[r,"\S_h"] \arrow[dr,"\mathrm{T} \S_h"'] & \sSet \arrow[d,"|\cdot|"] \\
     & \Top
     \end{tikzcd}
 \end{equation*}
\end{remark}

\begin{example}
Consider once more the standard~$2$--simplex with height function like in Example \ref{example:standard2}. The simplicial set~$(\Sh)_0$ is the disjoint union~$\Sh[0] \coprod \Sh[1] \coprod \Sh[2]$. All these components consist of a single point determined by the~$0$--simplices at the corresponding heights. If we don't count degeneracies, the simplicial set~$(\Sh)_1$ is the disjoint union~$\Sh[0,1] \coprod \Sh[1,2] \coprod \Sh[0,2]$. Again, all of the components are singletons corresponding to the~$1$--simplices~$e_2$,~$e_0$ and~$e_1$, respectively. Lastly,~$(\Sh)_2=\Sh[0,1,2]$, containing the~$2$--section corresponding to~$s$. In this example, the horizontal face maps of~$s$ corresponds to the ordinary face maps of the standard~$2$--simplex;~$d_0^h s=e_0$,~$d_1^h s=e_1$ and~$d_2^h s=e_2$. Notice how the higher section space~$(S_h)_2$ makes it possible to recover the topology of the~$2$--simplex. 
\end{example}
\begin{example}
Consider the product of two standard~$1$--simplices as in Example \ref{example:standard1timesstandard1}. 
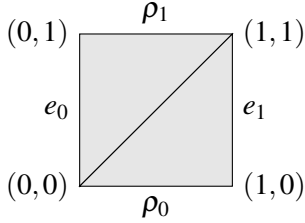
\begin{center}
    \begin{tikzpicture}[scale=2]
    \node[left] at (0,0) {$(0,0)$};
    \node[left] at (0,1) {$(0,1)$};
    \node[right] at (1,0) {$(1,0)$};
    \node[right] at (1,1) {$(1,1)$};
    \draw[fill=gray!20] (0,0) -- (1,1) -- (0,1) -- cycle;
    \draw[fill=gray!20] (0,0) -- (1,1) -- (1,0) -- cycle;
    
    \node[above] at (0.5,1) {$\rho_1$};
    \node[below] at (0.5,0) {$\rho_0$};
    \node[left] at (0,0.5) {$e_0$};
    \node[right] at (1,0.5) {$e_1$};
    \end{tikzpicture}
\end{center}
We obtain a height function by projecting the labels of the vertices to their first component~$h\colon (i,j)\mapsto i$.
The space of~$0$--sections is~$(\Sh)_0 = \Sh[0] \coprod \Sh[1]=h^{-1}0\coprod h^{-1}1$, with two components corresponding to the two~$1$--simplices~$e_0$ and~$e_1$. The space of~$1$--simplices is~$(\Sh)_1=\Sh[0,1]$. Consider the~$(1,1)$--section defined in terms of the identity~$\id_{\Delta^1 \times \Delta^1} \colon \Delta^1\times\Delta^1\rightarrow\Delta^1\times \Delta^1$. It has two horizontal faces~$e_0$ and~$e_1$ and two vertical faces given by the two~$1$--sections~$\rho_0$ and~$\rho_1$. We interpret this as~$\id_{\Delta^1 \times \Delta^1}$ being a homotopy from~$\rho_0$ to~$\rho_1$. 
\end{example}

\subsection{Proof of Theorem~\ref{intro:mainresult}}
\label{section:mainresult}
We will now prove that the diagonal of the section complex~$\S_h$, associated to a height function $h:X\to \R$, is homotopy equivalent to $X$ (Theorem ~\ref{intro:mainresult}). This justifies to use the spectral sequence of~$\S_h$ for extracting homological features of~$X$, which will thoroughly  discussed in Section~\ref{section:ReebGraphs}. %All details concerning this homological machinery is postponed until then. 
%We will now show that the spectral sequence associated to the bisimplicial set~$\Sh$ computes the homology of the simplicial set~$X$ by proving that the diagonal complex~$\diag\S_h$ is homotopy equivalent to~$X$. Details concerning homology computations are postponed til the next section.

The \emph{diagonal} of~$\S_h$,~$(\diag\S_h)_n$, is defined to have~$(n,n)$--sections~$\rho\colon \Delta^n\times \Delta^n\rightarrow X$ as~$n$--simplices. Since horizontal and vertical face maps are independent, we can safely define~$d_i=d_i^h d_i^v$ which is equal to~$d_i^vd_i^h$. Similarly,~$s_j=s_j^h s_j^v$. 

Let us understand how to relate~$\diag\S_h$ and~$X$: the~$n$--simplices in~$\diag \S_h$ define a subset of~$\map(\Delta^n,X)_n=\Map(\Delta^n \times \Delta^n,X)$, while the~$n$--simplices of~$X$ are given by the set~$\Map(\Delta^n,X)$. There are maps~$(\mathrm{id}_{\Delta^n},\mathrm{id}_{\Delta^n})\colon \Delta^n\rightarrow \Delta^n\times \Delta^n$,~$i\mapsto (i,i)$ and, conversely, the projection onto the "section component" is defined:~$\pr_0\colon \Delta^n\times \Delta^n\rightarrow \Delta^n$,~$(i,j)\mapsto i$. Pre-composition defines maps
\begin{equation*}
(\id,\id)^* \from \diag \Sh \to X \quad \text{and} \quad
\operatorname{pr}_0^* \from X \to \diag\Sh
\end{equation*} 
which will be proven mutual homotopy inverses. The latter map is well-defined. Indeed, if~$\tau \from \Delta^n \to X$ is any~$n$--simplex in~$X$, then the composition
\begin{equation*}
    \begin{tikzcd}
        \Delta^n \times \Delta^n \arrow[r,"\pr_0"] & \Delta^n \arrow[r,"\tau"] & X \arrow[r,"h"] & R
    \end{tikzcd}
\end{equation*} 
is in~$\Sh(\bar{a})$, where~$\bar{a}$ is defined by the~$n$--simplex~$h \circ \tau$.
Furthermore, these maps are clearly simplicial and the composition~$(\id,\id)^* \circ \operatorname{pr}_0^*$ is the identity. The proof of the theorem is thus reduced to finding a simplicial homotopy 
\begin{align*}
H\colon \id_{\diag \Sh} \Rightarrow \operatorname{pr}_0^* \circ (\id,\id)^*.
\end{align*}
To do so we first introduce for every~$n\geq 0$, two families of simplicial maps 
\begin{equation}\label{family1}
    \{\phi_{n,s}\from \Delta^n \times \Delta^n \to \Delta^n \times \Delta^n\}_{0\leq s \leq n+1} 
\end{equation}
and 
\begin{equation}\label{family2}
    \{\psi_{n,s}\from \Delta^n \times \Delta^n \to \Delta^n \times \Delta^n\}_{0\leq s \leq n+1}.
\end{equation}
Pulling these maps back along sections in~$(\diag \Sh)_n$ will then provide us with the components of our homotopy. Note that the parameter~$s$ will be necessary to make these components fit together into a simplicial map.

We specify how the maps (\ref{family1}) and (\ref{family2}) act on~$0$--simplices:
\begin{align} \label{def1}
    \phi_{n,s} (i,j) = \begin{cases}
        (i,i) \quad \text{if } i>n-s \quad \text{and } j\leq i \\
        (i,j) \quad \text{else}
    \end{cases} 
\end{align}
and 
\begin{align}\label{def2}
    \psi_{n,s} (i,j) = \begin{cases}
        (i,i) \quad \text{if } j\leq i \\
        (i,i) \quad \text{if } i<s \quad \text{and } j\geq i \\
        (i,j) \quad \text{else}
\end{cases} 
\end{align}
Note that the so defined assignments preserve the preorder on~$0$--simplices in~$\Delta^n \times \Delta^n$. 
Because~$\Delta^n \times \Delta^n$ is the nerve of the category~$[n] \times [n]$, and the nerve functor preserves products, (\ref{def1}) and (\ref{def2}) uniquely determine the families (\ref{family1}) and (\ref{family2}), respectively.

The following figure depicts the maps~$\phi_{1,s}$ and~$\psi_{1,s}$ in terms of their image.   

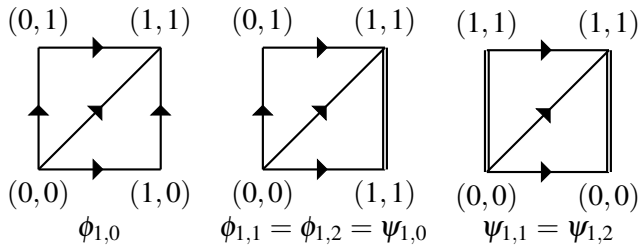
\begin{center}
        \begin{tikzpicture}[scale=0.2]
        
        \begin{scope}[thick, every node/.style={sloped,allow upside down}]
          \draw (-4,-4)-- node {\midarrow} (-4,4);
          \draw (-4,-4)-- node {\midarrow} (4,-4);
          \draw (-4,4)-- node {\midarrow} (4,4);  
          \draw (4,-4)-- node {\midarrow} (4,4);
          \draw (-4,-4)-- node {\midarrow} (4,4);
        
          \node [below] at (-4,-4) {$(0,0)$};
          \node [above] at (-4,4) {$(0,1)$};
          \node [above] at (4,4) {$(1,1)$};
          \node [below] at (4,-4) {$(1,0)$};

          \node at (0,-8) {$\phi_{1,0}$};
    \end{scope}
    \end{tikzpicture}
    \begin{tikzpicture}[scale=0.2]
        
        \begin{scope}[thick, every node/.style={sloped,allow upside down}]
            \draw (-4,-4)-- node {\midarrow} (-4,4);
            \draw (-4,-4)-- node {\midarrow} (4,-4);
            \draw (-4,4)-- node {\midarrow}(4,4);  
            \draw [-,double] (4,-4)-- (4,4);
            \draw (-4,-4)-- node {\midarrow} (4,4);
          
            \node [below] at (-4,-4) {$(0,0)$};
            \node [above] at (-4,4) {$(0,1)$};
            \node [above] at (4,4) {$(1,1)$};
            \node [below] at (4,-4) {$(1,1)$};
            
            \node at (0,-8) {$\phi_{1,1}=\phi_{1,2}=\psi_{1,0}$};
        \end{scope}
    \end{tikzpicture}
    \begin{tikzpicture}[scale=0.2]
        
        \begin{scope}[thick, every node/.style={sloped,allow upside down}]
            \draw (-4,-4)-- node {\midarrow} (4,-4);
            \draw [-,double] (4,-4)-- (4,4);
            \draw [-,double] (-4,-4)-- (-4,4);  
            \draw (-4,4)-- node {\midarrow} (4,4);
            \draw (-4,-4)-- node {\midarrow} (4,4);
          
            \node [below] at (-4,-4) {$(0,0)$};
            \node [above] at (-4,4) {$(1,1)$};
            \node [above] at (4,4) {$(1,1)$};
            \node [below] at (4,-4) {$(0,0)$};

            \node at (0,-8) {$\psi_{1,1}=\psi_{1,2}$};

        \end{scope}
    \end{tikzpicture}
    \end{center}

In general, the map~$\phi_{n,0}$ is always the identity and~$\psi_{n,n}=\psi_{n,n+1}=(\id,\id)\circ \pr_0$. These maps will respectively correspond to the start and end of our final homotopy.

We proceed by looking at how the maps~$\phi_{2,s}$ act on the three diagonal faces of~$\Delta^2 \times \Delta^2$. 
\newpage
$\operatorname{Im}(\delta^0\times \delta^0(\Delta^1 \times \Delta^1))$: 

    
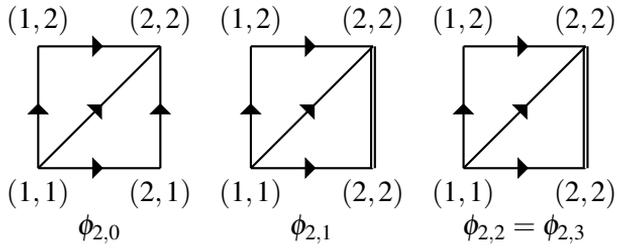
\begin{center}
        \begin{tikzpicture}[scale=0.2]
        
        \begin{scope}[thick, every node/.style={sloped,allow upside down}]
          \draw (-4,-4)-- node {\midarrow} (-4,4);
          \draw (-4,-4)-- node {\midarrow} (4,-4);
          \draw (-4,4)-- node {\midarrow} (4,4);  
          \draw (4,-4)-- node {\midarrow} (4,4);
          \draw (-4,-4)-- node {\midarrow} (4,4);
        
          \node [below] at (-4,-4) {$(1,1)$};
          \node [above] at (-4,4) {$(1,2)$};
          \node [above] at (4,4) {$(2,2)$};
          \node [below] at (4,-4) {$(2,1)$};

          \node at (0,-8) {$\phi_{2,0}$};

        %  \node at (-12,0) {$\operatorname{Im}(\delta^0\times \delta^0(\Delta^1 \times \Delta^1))$:};
    \end{scope}
    \end{tikzpicture}
    \begin{tikzpicture}[scale=0.2]
        
        \begin{scope}[thick, every node/.style={sloped,allow upside down}]
            \draw (-4,-4)-- node {\midarrow} (-4,4);
            \draw (-4,-4)-- node {\midarrow} (4,-4);
            \draw [-,double] (4,-4)-- (4,4);  
            \draw (-4,4)-- node {\midarrow} (4,4);
            \draw (-4,-4)-- node {\midarrow} (4,4);
          
            \node [below] at (-4,-4) {$(1,1)$};
            \node [above] at (-4,4) {$(1,2)$};
            \node [above] at (4,4) {$(2,2)$};
            \node [below] at (4,-4) {$(2,2)$};
            
            \node at (0,-8) {$\phi_{2,1}$};

        \end{scope}
    \end{tikzpicture}
    \begin{tikzpicture}[scale=0.2]
        
        \begin{scope}[thick, every node/.style={sloped,allow upside down}]
            \draw (-4,-4)-- node {\midarrow} (-4,4);
            \draw (-4,-4)-- node {\midarrow} (4,-4);
            \draw [-,double] (4,-4)-- (4,4);  
            \draw (-4,4)-- node {\midarrow} (4,4);
            \draw (-4,-4)-- node {\midarrow} (4,4);
          
            \node [below] at (-4,-4) {$(1,1)$};
            \node [above] at (-4,4) {$(1,2)$};
            \node [above] at (4,4) {$(2,2)$};
            \node [below] at (4,-4) {$(2,2)$};
            
            \node at (0,-8) {$\phi_{2,2}=\phi_{2,3}$};

        \end{scope}
    \end{tikzpicture}
    \end{center}
$\operatorname{Im}(\delta^1\times \delta^1(\Delta^1 \times \Delta^1))$:
    
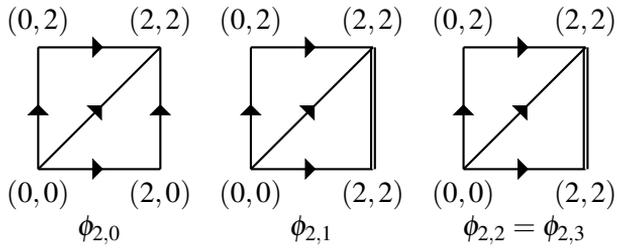
\begin{center}
        \begin{tikzpicture}[scale=0.2]
        
        \begin{scope}[thick, every node/.style={sloped,allow upside down}]
          \draw (-4,-4)-- node {\midarrow} (-4,4);
          \draw (-4,-4)-- node {\midarrow} (4,-4);
          \draw (-4,4)-- node {\midarrow} (4,4);  
          \draw (4,-4)-- node {\midarrow} (4,4);
          \draw (-4,-4)-- node {\midarrow} (4,4);
        
          \node [below] at (-4,-4) {$(0,0)$};
          \node [above] at (-4,4) {$(0,2)$};
          \node [above] at (4,4) {$(2,2)$};
          \node [below] at (4,-4) {$(2,0)$};
          
          \node at (0,-8) {$\phi_{2,0}$};

%          \node at (-12,0) {$\operatorname{Im}(\delta^1\times \delta^1(\Delta^1 \times \Delta^1))$:};
    \end{scope}
    \end{tikzpicture}
    \begin{tikzpicture}[scale=0.2]
        
        \begin{scope}[thick, every node/.style={sloped,allow upside down}]
           \draw (-4,-4)-- node {\midarrow} (-4,4);
            \draw (-4,-4)-- node {\midarrow} (4,-4);
            \draw [-,double] (4,-4)-- (4,4);  
            \draw (-4,4)-- node {\midarrow} (4,4);
            \draw (-4,-4)-- node {\midarrow} (4,4);
          
            \node [below] at (-4,-4) {$(0,0)$};
            \node [above] at (-4,4) {$(0,2)$};
            \node [above] at (4,4) {$(2,2)$};
            \node [below] at (4,-4) {$(2,2)$};
            
            \node at (0,-8) {$\phi_{2,1}$};

        \end{scope}
    \end{tikzpicture}
    \begin{tikzpicture}[scale=0.2]
        
        \begin{scope}[thick, every node/.style={sloped,allow upside down}]
            \draw (-4,-4)-- node {\midarrow} (-4,4);
            \draw (-4,-4)-- node {\midarrow} (4,-4);
            \draw [-,double] (4,-4)-- (4,4);  
            \draw (-4,4)-- node {\midarrow} (4,4);
            \draw (-4,-4)-- node {\midarrow} (4,4);
          
            \node [below] at (-4,-4) {$(0,0)$};
            \node [above] at (-4,4) {$(0,2)$};
            \node [above] at (4,4) {$(2,2)$};
            \node [below] at (4,-4) {$(2,2)$};
            
            \node at (0,-8) {$\phi_{2,2}=\phi_{2,3}$};

        \end{scope}
    \end{tikzpicture}
    \end{center}
$\operatorname{Im}(\delta^2\times \delta^2(\Delta^1 \times \Delta^1))$:
    
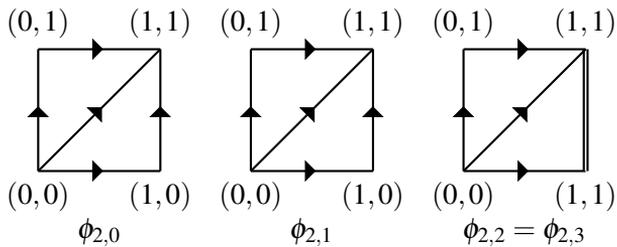
\begin{center}
        \begin{tikzpicture}[scale=0.2]
        
        \begin{scope}[thick, every node/.style={sloped,allow upside down}]
          \draw (-4,-4)-- node {\midarrow} (-4,4);
          \draw (-4,-4)-- node {\midarrow} (4,-4);
          \draw (-4,4)-- node {\midarrow} (4,4);  
          \draw (4,-4)-- node {\midarrow} (4,4);
          \draw (-4,-4)-- node {\midarrow} (4,4);
        
          \node [below] at (-4,-4) {$(0,0)$};
          \node [above] at (-4,4) {$(0,1)$};
          \node [above] at (4,4) {$(1,1)$};
          \node [below] at (4,-4) {$(1,0)$};

%          \node at (-12,0) {$\operatorname{Im}(\delta^2\times \delta^2(\Delta^1 \times \Delta^1))$:};

          \node at (0,-8) {$\phi_{2,0}$};
    \end{scope}
    \end{tikzpicture}
    \begin{tikzpicture}[scale=0.2]
        
        \begin{scope}[thick, every node/.style={sloped,allow upside down}]
            \draw (-4,-4)-- node {\midarrow} (-4,4);
            \draw (-4,-4)-- node {\midarrow} (4,-4);
            \draw (-4,4)-- node {\midarrow} (4,4);  
            \draw (4,-4)-- node {\midarrow} (4,4);
            \draw (-4,-4)-- node {\midarrow} (4,4);
          
            \node [below] at (-4,-4) {$(0,0)$};
            \node [above] at (-4,4) {$(0,1)$};
            \node [above] at (4,4) {$(1,1)$};
            \node [below] at (4,-4) {$(1,0)$};

            \node at (0,-8) {$\phi_{2,1}$};
        \end{scope}
    \end{tikzpicture}
    \begin{tikzpicture}[scale=0.2]
        
        \begin{scope}[thick, every node/.style={sloped,allow upside down}]
            \draw (-4,-4)-- node {\midarrow} (-4,4);
            \draw (-4,-4)-- node {\midarrow} (4,-4);
            \draw [-,double] (4,-4)-- (4,4);  
            \draw (-4,4)-- node {\midarrow} (4,4);
            \draw (-4,-4)-- node {\midarrow} (4,4);
          
            \node [below] at (-4,-4) {$(0,0)$};
            \node [above] at (-4,4) {$(0,1)$};
            \node [above] at (4,4) {$(1,1)$};
            \node [below] at (4,-4) {$(1,1)$};

            \node at (0,-8) {$\phi_{2,2}=\phi_{2,3}$};

        \end{scope}
    \end{tikzpicture}
    \end{center}
    
The key observation to be made from looking at these pictures is that on the diagonal face~$(\delta^l \times \delta^l)(\Delta^1 \times \Delta^1)$ the map~$\phi_{2,s}$ is determined by~$\phi_{1,s}$ if~$l \leq 2-s$ and by~$\phi_{1,s-1}$ if~$l>2-s$. This pattern generalizes to all dimensions~$n$, also for the maps~$\psi_{n,s}$. With this insight in mind we prove the following lemma.   
\newpage      

\begin{lemma}\label{deltalemma}
    The family of simplicial maps~$\{\phi\}_{n,s}$ satisfies for every~$n \geq 1$
    \begin{equation}\label{firstdiagram}
        \begin{tikzcd}
            \Delta^{n-1} \times \Delta^{n-1} \arrow[r,"\delta^l\times \delta^l"] \arrow[d,"\phi_{n-1,s}"'] & \Delta^n \times \Delta^n \arrow[d,"\phi_{n,s}"] \\
            \Delta^{n-1} \times \Delta^{n-1} \arrow[r,"\delta^l \times \delta^l"] & \Delta^n \times \Delta^n ,
        \end{tikzcd}
        \quad     \begin{tikzcd}
            \Delta^{n} \times \Delta^{n} \arrow[r,"\sigma^l\times \sigma^l"] \arrow[d,"\phi_{n,s}"'] & \Delta^{n-1} \times \Delta^{n-1} \arrow[d,"\phi_{n-1,s}"] \\
            \Delta^{n} \times \Delta^{n} \arrow[r,"\sigma^l \times \sigma^l"] & \Delta^{n-1} \times \Delta^{n-1} 
        \end{tikzcd} \quad \text{for } l \leq n-s
    \end{equation}
    and
    \begin{equation}\label{seconddiagram}
        \begin{tikzcd}
            \Delta^{n-1} \times \Delta^{n-1} \arrow[r,"\delta^l\times \delta^l"] \arrow[d,"\phi_{n-1,s-1}"'] & \Delta^n \times \Delta^n \arrow[d,"\phi_{n,s}"] \\
            \Delta^{n-1} \times \Delta^{n-1} \arrow[r,"\delta^l \times \delta^l"] & \Delta^n \times \Delta^n 
        \end{tikzcd}
    \quad
        \begin{tikzcd}
            \Delta^n \times \Delta^n \arrow[r,"\sigma^l\times \sigma^l"] \arrow[d,"\phi_{n,s}"']  & \Delta^{n-1} \times \Delta^{n-1} \arrow[d,"\phi_{n-1,s-1}"] \\
             \Delta^n \times \Delta^n  \arrow[r,"\sigma^l \times \sigma^l"] & \Delta^{n-1} \times \Delta^{n-1}
        \end{tikzcd} \quad \text{for } l > n-s
    \end{equation}
    and likewise for the family~$\{\psi\}_{n,s}$.
\end{lemma}

\begin{proof}
    The proof is straightforward: We compute the images of~$0$--simplices along both sides of the asserted diagrams. As previously mentioned, this suffices because morphisms between products of standard simplices are uniquely determined by the image of~$0$--simplices.
    We first consider the left-hand diagram in (\ref{firstdiagram}). It should commute whenever~$l \leq n-s$. An arbitrary~$0$--simplex~$(i,j)$ is mapped to
    \begin{align*}
        (\delta^l \times \delta^l) \circ \phi_{n-1,s} = \begin{cases}
            (\delta^l \times \delta^l)(i,i) \quad \text{if } i > n-1-s \text{ and } j \leq i \\
            (\delta^l \times \delta^l)(i,j) \quad \text{else}
        \end{cases}
    \end{align*}
    along the lower left composition and 
    \begin{align*}
        \phi_{n,s} \circ (\delta^l \times \delta^l) = 
        \begin{cases}
            (\delta^l \times \delta^l)(i,i) \quad \text{if } \delta^l(i) > n-s \text{ and } \delta^l(j) \leq \delta^l(i)  \\
            (\delta^l \times \delta^l)(i,j) \quad \text{else}
        \end{cases}
    \end{align*}
    along the upper right composition. Observe how the inequalities~$j \leq i$ and~$\delta^l(j) \leq \delta^l(i)$ are equivalent. Moreover, the inequality~$i > n-1-s$ is equivalent to~$\delta^l(i) > n-s$. Indeed, if~$i > n-1-s$, then~$l\leq n-s\leq i$ so that~$\delta^l(i) > n-s$. Conversely, if~$\delta^l(i)>n-s$, then clearly~$i>n-s-1$. Hence the first diagram commutes.  
    
    The commutativity of all the other diagrams is shown in the same manner. No complications arise in the corresponding computations and we will thus not spell them out here.
\end{proof}

With this recursive description of the families~$\{\phi_{n,s}\}$ and~$\{\psi_{n,s}\}$ we can now give the homotopy  
\begin{align*}
   H\colon \id_{\diag \Sh} \Rightarrow \operatorname{pr}_1^* \circ (\id,\id)^*,
\end{align*}
which finishes the proof. 

%We restate the main result as follows.
%(\ref{homotopy}) and in doing so prove that~$\diag S_h \cong X$.

%\begin{theorem}[Theorem~\ref{intro:mainresult}]
%    Let~$h \from X \to R$ be a height function on the simplicial set~$X$. There is a simplicial homotopy~$H\colon\id_{\diag \Sh} \Rightarrow \operatorname{pr}_1^* \circ (\id,\id)^*$. In particular, the diagonal of the section complex~$\diag \Sh$ is homotopy equivalent to~$X$:
 %   \begin{align*}
 %       \diag\Sh \cong X.
 %   \end{align*}
 %   \end{theorem}
    \begin{proof}[Proof of Theorem~\ref{intro:mainresult}]
        Our model for the interval will be the~$(2,2)$--horn~$\Lambda^2_2$:~$0\rightarrow 2 \leftarrow 1$. Note that an~$n$--simplex in~$\Lambda^2_2$ is equivalent to a map~$m$ from~$\{0,1,\dots,n\}$ to either~$\{0,2\}$ or~$\{1,2\}$, respecting the ordering. In the first case, we use the notation~$(0:n-s+1,2:s)$, counting the number of times~$m$ meets~$0$ and~$2$. Dually,~$(2:n-s+1,1:s)$ is used in the second case.

        The components of the asserted homotopy are given by 
        \begin{align*}
            H_n \from (\diag\Sh)_n \times (\Lambda^2_2)_n \to (\diag\Sh)_n
        \end{align*}
        where 
        \begin{align*} 
            H_n(\rho,t) = \begin{cases}
                                \rho \circ \phi_{n,s} \quad t=(0:n-s+1,2:s) \\
                                \rho \circ \psi_{n,s} \quad t=(2:n-s+1,1:s)
                            \end{cases}
        \end{align*}
        For this to constitute a simplicial map, there must be commutative diagrams
        \begin{equation*}
            \begin{tikzcd}
                (\diag\Sh)_n \times (\Lambda^2_2)_n  \arrow[r,"H_n"] \arrow[d,"d_l"] & (\diag\Sh)_n \arrow[d,"d_l"] \\ 
                (\diag\Sh)_{n-1} \times (\Lambda^2_2)_{n-1} \arrow[r,"H_{n-1}"] & (\diag\Sh)_{n-1}
            \end{tikzcd}
        \end{equation*}
        whenever~$0 \leq l \leq n$ and
        \begin{equation*}
            \begin{tikzcd}
                (\diag\Sh)_{n-1} \times (\Lambda^2_2)_{n-1}  \arrow[r,"H_{n-1}"] \arrow[d,"s_l"] & (\diag\Sh)_{n-1} \arrow[d,"s_l"] \\ 
                (\diag\Sh)_{n} \times (\Lambda^2_2)_{n} \arrow[r,"H_{n}"] & (\diag\Sh)_{n}
            \end{tikzcd}
        \end{equation*}
        whenever~$0 \leq l < n$. We will only verify the case~$t=(0:n-p+1,2:p)$. This is because of how~$t=(2:n-p+1,1:p)$ is completely analogous. The upper right composition in the first diagram is:
        \begin{align*}
            d_l H_n(\rho,(0:n-s+1,2:s)) 
            =&d_l(\rho \circ \phi_{n,s}) \\
            =&\rho \circ \phi_{n,s} \circ (\delta^l \times {\delta^l}) 
        \end{align*}    
        Hence, we deduce
        \begin{align*} \label{homotopy}
            d_lH_n(\rho,t) = \begin{cases}
                                \rho \circ (\delta^l \times \delta^l) \circ \phi_{n-1,s} \quad l\leq n-s \\
                                \rho \circ (\delta^l \times \delta^l) \circ \phi_{n-1,s-1} \quad l>n-s
                            \end{cases}
        \end{align*}
        due to the left hand diagrams~(\ref{firstdiagram}) and~(\ref{seconddiagram}) given in Lemma~\ref{deltalemma}. 
        Since
        \begin{align*}
            d_l(\rho,t) = \begin{cases}
                               (d_l\rho,(0:n-s,2:s)) \quad l\leq n-s \\
                                (d_l\rho,(0:n-s+1,2:s-1)) \quad l>n-s,
                            \end{cases}
        \end{align*}
        we have~$d_l H_n(\rho,t)=H_{n-1}d_l(\rho,t)$. This establishes the commutativity of the first diagram.
        
        Using the right-hand diagrams in (\ref{firstdiagram}) and (\ref{seconddiagram}) we can show compatibility with the degeneracy maps in the same way. This concludes the construction of the homotopy and thus the proof. 
        \end{proof}

\section{The section spectral sequence}
\label{section:ReebGraphs}

We apply homology to the section spaces~$(\S_h)_p$ and assemble the resulting vector spaces into chain complexes that we term with the name \textit{Reeb complexes}. Due to Corollary \ref{corollary:reebcomplexes} of Section 4, these Reeb complexes serve as a completely combinatorical model for the Reeb complexes associated to continuous height functions, as defined in \cite{Vaupel_TDA_2022}. As a bisimplicial set, the section complex has an associated spectral sequence. We name it \textit{section spectral sequence}.  Theorem \ref{intro:mainresult} then implies that the section spectral sequence computes the homology of the height functions base space in general. The Reeb complexes can be extracted from the first page of the section spectral sequence (Proposition \ref{prop:spectralsequence}) and thus provide a first order approximation of the homology of~$X$ in terms of the homology of section spaces.     

\subsection{Reeb Complexes}
\label{section:ReebComplexes}

%For a height function~$h\colon X \rightarrow \R$ there is the associated section complex~$\S_h$ which is a functor
%\[
%\Delta^{\op} \rightarrow \sSet
%\]
%whose space of~$p$-simplices is~$(\S_h)_p=\bigsqcup \Sh(\bar{a})$, indexed over all sequences~$\bar{a}$ in~$\R_p$; non-decreasing sequences %in~$\mathbb{R}$ of length~$p+1$. 

Recall that the section complex,~$\S_h$, consists of all section spaces~$(\S_h)_p$. We can thus apply any homology functor~$\homology_q$ to~$(\S_h)_p$ and induce~$\homology_q d^h_i\colon \homology_q (\S_h)_p\rightarrow \homology_q(\S_h)_{p-1}$. This defines a simplicial vector space~$\homology_q\S_h$, because every set of~$p$--simplices,~$\homology_q(\S_h)_p$, is a vector space. Furthermore, for any simplicial vector space~$V$, there is a complex~$\mathrm{C}V$, called the \emph{Moore Complex}. Its~$p$th entry~$\mathrm{C} V_p$ is equal to the vector space~$V_p$, and its differential is induced by the alternating sum of face maps,~$\partial=\sum (-1)^i d_i$. Denote by~$\mathrm{D}V$ the subcomplex of~$\mathrm{C}V$ whose~$p$th entry only consists of the degenerate~$p$--simplices in~$V_p$. The differential induces a well-defined differential~$\mathrm{C}V_p/\mathrm{D}V_p\rightarrow \mathrm{C}V_{p-1}/\mathrm{D}V_{p-1}$ from which we define the \emph{non-degenerate complex}~$\mathrm{C}V/\mathrm{D}V$.

\begin{definition}
    For a height function~$h\colon X \rightarrow \R$ and integer~$q\geq 1$, we define the~$q$th \emph{Reeb complex}~$\mathcal{G}_q$ as the chain complex~$\mathrm{C}(\homology_q\S_h)/ \mathrm{D}(\homology_q \S_h)$.
\end{definition}
\begin{example}
\label{example:explicitReebComplex}
The~$q$th Reeb complex associated to a height function~$h\colon X \rightarrow \R$ has
\[
(\mathcal{G}_q)_p=\homology_q\left(\bigsqcup \Sh(\bar{a})\right)\simeq \bigoplus \homology_q \S_h[\bar{a}]
\]
as its~$p$th entry, ranging over all~\emph{increasing} sequences~$\bar{a}=(a_0,\dots,a_p)$ in~$\R_p$. 
\end{example}

Reeb complexes provide an approximate tool to better understand homological features of the underlying space~$X$. This is achieved by understanding how homology generators flow between height levels and along sections.

\begin{example}
\label{example:sh}
Recall the standard~$2$--simplex~$\Delta^2$, with heights as indicated by the labels:
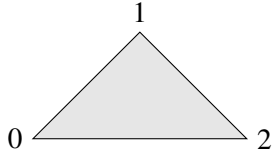
\begin{center}
\begin{tikzpicture}[scale= 0.7]

    \node[left] at (0,0) {$0$};
    \node[above] at (2,2) {$1$};
    \node[right] at (4,0) {$2$};
    
    \draw[fill=gray!20] (0,0) -- (2,2) -- (4,0) -- cycle;
    
\end{tikzpicture}
\end{center}
We glue two copies of~$\Delta^2$ together along their boundary~$\partial\Delta^2$ to obtain~$\Delta^2\coprod_{\partial\Delta^2}\Delta^2$, a simplicial model for the~$2$--sphere. Label~$0$ and~$1$--simplices by their integers, e.g.~$01$ is the~$1$--simplex from~$0$ to~$1$. The two~$2$--simplices sharing a common boundary, are denoted~$a$ and~$b$. 
%There is an evident height function~$h\colon \Delta^2\coprod_{\partial\Delta^2}\Delta^2\rightarrow \R$ induced from the inclusion~$\Delta^2\hookrightarrow \R$, as indicated in the above picture from left to right. 
To determine a basis for~$\homology_q\S_h$, we identify the homotopy types of all section spaces, indexed by increasing sequences:

\begin{center}
 \begin{tabular}{ |c | ccccccc | } 
 \hline
 $\bar{a}$ & $(0)$ & $(1)$ & $(2)$ & $(0,1)$ & $(0,2)$ & $(1,2)$ & $(0,1,2)$\\ 
 \hline
 $\S_h[\bar{a}]$ & $\{0\}$ & $\{1\}$ & $\{2\}$ & $\{01\}$ & $\{02\}$ & $\{12\}$ & $\{a, b\}$\\ 
 \hline
 Homotopy type & pt & pt & pt & pt & pt & pt & $\text{pt}\coprod \text{pt}$\\ 
 \hline
\end{tabular}
\end{center}
%The pre-image~$\S_h(0)=h^{-1} 0$ only consists of a single~$0$--simplex, whereas
Hence, all Reeb complexes with~$q \geq 1$ are trivial. For~$q=0$, however, we determine the boundary maps
\[
\partial_1\colon \oplus \homology_0 (\Sh)_1 \rightarrow \oplus \homology_0 (\Sh)_0 \text{  and  } \partial_2 \colon \homology_0 (\Sh)_2 \rightarrow \oplus \homology_0 (\Sh)_1.
\]
Picking the evident bases from the above table yields
\[
\partial_1 = 
\begin{bmatrix}
    -1 & -1  & 0   \\
    1  & 0 & -1  \\
    0  & 1  &  1  \\
\end{bmatrix}
\text{ and }
\partial_2 = 
\begin{bmatrix}
    1 & 1  \\
    -1  & -1 \\
    1  & 1  \\
\end{bmatrix}
\]
in coordinates. As an example, the first column of~$\partial_1$ is obtained by applying the target~$d_0^h$ and source~$d_1^h$ to generators in~$\homology \S_h[0,1]$:~$\partial_1[01]=[1]-[0]$. Hence, we can present
\[
\mathcal{G}_0 \colon \; k^3 \xleftarrow{\partial_1} k^3 \xleftarrow{\partial_2} k^2.
\]
Elementary linear algebra gives~$\homology_0 \mathcal{G}_0 = k$ and~$\homology_2 \mathcal{G}_0 = k$, whereas other homology groups are trivial. In this particular example, the zeroth Reeb complex carries the homology of the underlying space~$\Delta^2\coprod_{\partial\Delta^2} \Delta^2$.
\end{example}

\begin{definition}
\label{definition:subdivided}
We say that a simplicial set~$X$ is~\emph{subdivided} according to a height function~$h\colon X\rightarrow \R$ if all section spaces~$\S_h[a,b]$ are empty whenever there is an intermediate height level~$a<c<b$.
\end{definition}
Whenever~$X$ is subdivided with respect to a height function~$h\colon X\rightarrow \R$, the Reeb complexes only have two non-zero entries~$(\mathcal{G}_q)_p$. Indeed, the~$p$th entry,~$\oplus \S_h[a_0,\dots,a_p]$, of the formula in Example~\ref{example:explicitReebComplex}, is zero for~$p\geq 2$. Thus, choosing coordinates in this case reduces the information contained in~$\mathcal{G}_q$ to a single matrix. Interpreting this matrix as an \emph{incidence matrix} provides a graph which gives insight as to how homology generators flow across height levels. This is illustrated with an example.
\begin{example}
\label{example:subdivision}
We subdivide~$\Delta^2$ according to the heights given in Example~\ref{example:sh}:
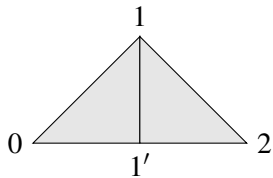
\begin{center}
\begin{tikzpicture}[scale= 0.7]

    \node[left] at (0,0) {$0$};
    \node[above] at (2,2) {$1$};
    \node[below] at (2,0) {$1'$};
    \node[right] at (4,0) {$2$};
    
    \draw[fill=gray!20] (0,0) -- (2,2) -- (2,0) -- cycle;
    \draw[fill=gray!20] (2,2) -- (2,0) -- (4,0) -- cycle;
\end{tikzpicture}
\end{center}

The subdivided~$2$--simplex still maps to~$\R$ by also sending~$1'\mapsto 1$ in~$\R_0$. We construct a space as in the previous example, by gluing two copies of the subdivided~$2$--simplex together along the boundary defined by the cycle~$0\rightarrow 1\rightarrow 2 \leftarrow 1' \leftarrow 0$. It is not difficult to determine the homotopy types of the section spaces:

\begin{center}
 \begin{tabular}{ |c | ccccccc | } 
 \hline
 $\bar{a}$ & $(0)$ & $(1)$ & $(2)$ & $(0,1)$ & $(0,2)$ & $(1,2)$ & $(0,1,2)$\\ 
 \hline
 $\S_h[\bar{a}]$ & pt & $S^1$ & pt & $S^1$ & $\emptyset$ & $S^1$ & $\emptyset$\\ 
 \hline
\end{tabular}
\end{center}

For instance, the homotopy type of~$\S_h[0,1]$ is given as follows. There are two~$0$--sections represented by the edges~$01$ and~$0 1'$. Each copy of the subdivided~$2$--simplex provides a~$(1,1)$-section, i.e. a homotopy
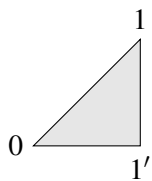
\begin{center}
\begin{tikzpicture}[scale= 0.7]

    \node[left] at (0,0) {$0$};
    \node[above] at (2,2) {$1$};
    \node[below] at (2,0) {$1'$};
    
    \draw[fill=gray!20] (0,0) -- (2,2) -- (2,0) -- cycle;

\end{tikzpicture}
\end{center}

between the two sections~$01$ and~$01'$, but no higher simplices connect them. Thus,~$\S_h[0,1]$ is isomorphic to two~$1$--simplices glued tail to tail and head to head. The horizontal face maps used to calculate Reeb complexes can be depicted:
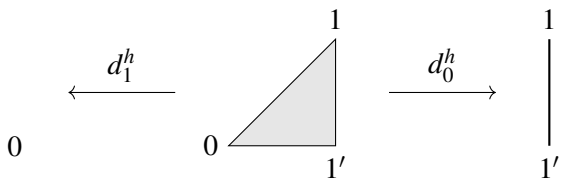
\begin{center}
\begin{tikzpicture}[scale= 0.7]

    \node[left] at (0,0) {$0$};
    \node[above] at (2,2) {$1$};
    \node[below] at (2,0) {$1'$};
    
    \draw[fill=gray!20] (0,0) -- (2,2) -- (2,0) -- cycle;
    
    \node[above] at (6,2) {$1$};
    \node[below] at (6,0) {$1'$};
    
    \draw[thick] (6,2)--(6,0);
    \draw[<-] (5,1) -- (3,1);
    \node[above] at (4,1) {$d_0^h$};
    
    \node[] at (-4,0) {0};
    \draw[<-] (-3,1) -- (-1,1);
    \node[above] at (-2,1) {$d_1^h$};
\end{tikzpicture}
\end{center}
This translates to~$\homology_1\S_h[0,1] \xrightarrow{0} \homology_1\S_h[0]$ and~$\homology_1\S_h[0,1] \xrightarrow{1} \homology_1\S_h[1]$ on~$\homology_1$.
We calculate the two non-trivial Reeb complexes in coordinates
\[
\mathcal{G}_0\colon \; k^3 \xleftarrow{\partial} k^2, \text{ with } \partial = 
\begin{bmatrix}
    -1 & 0   \\
    1  & -1  \\
    0  & 1   \\
\end{bmatrix}
\]
and
\[
\mathcal{G}_1\colon \;k \xleftarrow{\partial'} k^2,\text{ with } \partial' = 
\begin{bmatrix}
    1 & -1
\end{bmatrix}.
\]
To draw the associated graphs, we think of the basis elements in~$(\mathcal{G}_0)_0$ as vertices whereas the basis in~$(\mathcal{G}_0)_1$ defines edges. Then the first column in~$\partial$ tells us that the edge given by the first basis vector in~$k^2$ connects the first and second vertices (basis elements) in~$k^3$. For~$\mathcal{G}_1$ we have to take a bit care as we have two edges and one vertex. One edge starts in the vertex, the other ends in it. This is due to face maps being sent to zero maps in~$\homology_1$, a phenomenon that does not occur for~$\homology_0$. We can assemble this information in a barcode-like diagram:
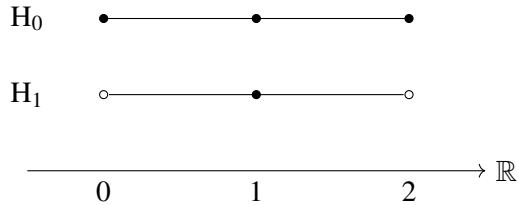
\begin{center}
\begin{tikzpicture}[main_node/.style={circle, fill=black!100, scale= 0.2}]

    \filldraw (0,0) circle (1.5pt);
    \filldraw (2,0) circle (1.5pt);
    \filldraw (4,0) circle (1.5pt);
  %  \node[main_node] (r01) at (0, 0) {0};
  %  \node[main_node] (r02) at (2, 0)  {1};
  %  \node[main_node] (r03) at (4, 0) {2};
    \node[] at (-1,0) {$\homology_0$};
    \draw (0,0) -- (4,0);
    
  %  \node[main_node] (r11) at (2,-1) {1};
    \filldraw (2,-1) circle (1.5pt);
    \draw (0,-1) circle (1.5pt);
    \draw (4,-1) circle (1.5pt);
    \node[] at (-1,-1) {$\homology_1$};
    \draw[] (0.07,-1) -- (3.93,-1);
    
    \draw[->] (-1,-2) -- (5,-2);
    
    \node[below] at (0,-2) {$0$};
    \node[below] at (2,-2) {$1$};
    \node[below] at (4,-2) {$2$};
    \node[right] at (5,-2) {$\mathbb{R}$};
\end{tikzpicture}
\end{center}
\end{example}

In the previous example, the reader familiar with Reeb graphs may have observed that the graph determined by introducing coordinates to~$\mathcal{G}_0$ is the Reeb graph of the given height function. This observation is true in general, if~$X$ is subdivided according to~$h\colon X\rightarrow \R$.
\begin{proposition}
\label{proposition:isReebGraph}
If~$X$ is subdivided according to a height function~$h\colon X\rightarrow \R$. Then the simplicial set~$\pi_0 \S_h$ is the Reeb graph of~$h$. In particular, the zeroth Reeb complex~$\mathcal{G}_0$ computes the homology of the associated Reeb graph.
\end{proposition}
\begin{proof}
The result is an immediate consequence of Proposition~\ref{proposition:comparisonWithTopological}, to be proved in Section~\ref{section:topologicalspaces!?}, and Theorem~$1.2$ in~\cite{trygsland2021}.
\end{proof}

\subsection{Background on spectral sequences}
\label{section:ss}
In the present Section \ref{section:ss} and the following Section \ref{section:sectionspectralsequence} we make the simplifying assumption of taking homology with field coefficients. Everything works as well for arbitrary generalised homology theories and our only aim is a slight simplification of the presentation.\\

A double chain complex~$C$ is a collection~$C_{p,q}$ of vector spaces together with horizontal and vertical boundary maps~$\partial_h\colon C_{p,q}\rightarrow C_{p-1,q}$ and~$\partial_v\colon C_{p,q} \rightarrow C_{p,q-1}$. The maps are further required to satisfy~$\partial_h^2=0$,~$\partial_v^2=0$ and~$\partial_v\partial_h=\partial_h\partial_v$. We always assume a double chain complex to be contained within the first quadrant, so that all entries with~$p$ or~$q$ negative are zero. To a double complex~$C$, we can functorially associate a chain complex~$\Tot C$, the \emph{total complex} of~$C$ with~$\Tot C_n=\oplus_{p+q=n} C_{p,q}$.

There is a functor~$\F\colon \ssSet \rightarrow \dCh$. It sends a bisimplicial set~$X$ to the double complex~$\F X$ with~$(\F X)_{p,q}=\F X_{p,q}$, the free vector space on~$X_{p,q}$. The horizontal and vertical boundary maps are induced by the horizontal and vertical face maps:~$\partial_h=\sum (-1)^i d_i^h$ and~$\partial_v=\sum (-1)^i d_i^v$. Total complexes thus define a functor~$\ssSet\rightarrow \Ch$, by mapping a bisimplicial set~$X$ to the total complex~$\Tot\F X$. A theorem of Dold and Puppe~\cite{dold1961homologie,goerss2009simplicial} tells us that~$\Tot \F X$ is naturally homology equivalent to~$\diag X$, the diagonal on~$X$:
\[
\homology_{\ast} \Tot \F X \simeq \homology_\ast \diag X.
\]
Therefore in order to understand the homology of~$\diag X$, one may rather consider the homology~$\Tot \F X$. One advantage of the total complex, is that it comes with a~\emph{spectral sequence} for computing its homology. The following is a brief recap of how this computational tool works. We refer to~\cite{mccleary2001user} for a more in-depth introduction.

Given a double complex~$C$, we define the zeroth page of the spectral sequence~$\page^0_{p,q}=C_{p,q}$ and remember only the vertical boundary maps~$\partial_v=\partial^0$:
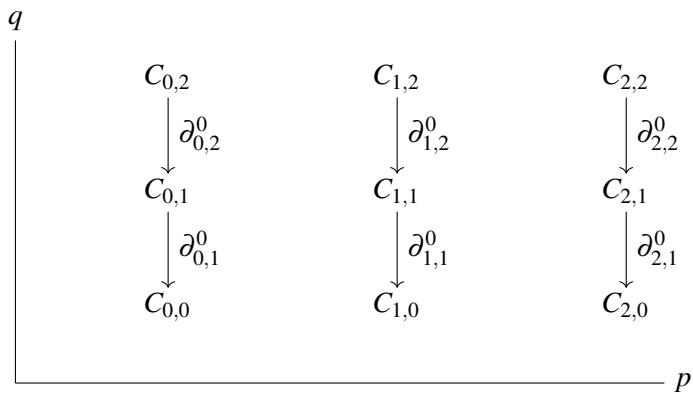
\begin{center}
\begin{tikzpicture}
	\draw (0,0) -- (8.5,0);
	\draw (0,0) -- (0,4.5);
	
	\draw[->] (2,2.25) -- (2,1.25);
	\draw[->] (2,3.75) -- (2,2.75);
	
	\draw[->] (5,2.25) -- (5,1.25);
	\draw[->] (5,3.75) -- (5,2.75);
	
	\draw[->] (8,2.25) -- (8,1.25);
	\draw[->] (8,3.75) -- (8,2.75);
	
	\node [right] at (2,1.75) {$\partial^0_{0,1}$};
	\node [right] at (2,3.25) {$\partial^0_{0,2}$};
	
	\node [right] at (5,1.75) {$\partial^0_{1,1}$};
	\node [right] at (5,3.25) {$\partial^0_{1,2}$};
	
	\node [right] at (8,1.75) {$\partial^0_{2,1}$};
	\node [right] at (8,3.25) {$\partial^0_{2,2}$};
	
	\node [] at (2,1) {$C_{0,0}$};
	\node [] at (2,2.5) {$C_{0,1}$};
	\node [] at (2,4) {$C_{0,2}$};
	
	\node [] at (5,1) {$C_{1,0}$};
	\node [] at (5,2.5) {$C_{1,1}$};
	\node [] at (5,4) {$C_{1,2}$};
	
	\node [] at (8,1) {$C_{2,0}$};
	\node [] at (8,2.5) {$C_{2,1}$};
	\node [] at (8,4) {$C_{2,2}$};

	\node [right] at (8.5,0) {$p$};
	\node [above] at (0,4.5) {$q$};
\end{tikzpicture}
\end{center}

Applying homology produces the first page~$\page^1_{p,q}=\homology_q C_{p,q}$ with induced differentials~$\partial^1$ from the horizontal differentials of~$C$.

\begin{center}
\begin{tikzpicture}
	\draw (0,0) -- (8.5,0);
	\draw (0,0) -- (0,4.5);
	
	\draw[->] (4.25,1) -- (2.75,1);
	\draw[->] (4.25,2.5) -- (2.75,2.5);
	\draw[->] (4.25,4) -- (2.75,4);
	
	\draw[->] (7.25,1) -- (5.75,1);
	\draw[->] (7.25,2.5) -- (5.75,2.5);
	\draw[->] (7.25,4) -- (5.75,4);
	
	\node [above] at (3.5,1) {$\partial^1_{1,0}$};
	\node [above] at (3.5,2.5) {$\partial^1_{1,1}$};
	\node [above] at (3.5,4) {$\partial^1_{1,2}$};
	
	\node [above] at (6.5,1) {$\partial^1_{2,0}$};
	\node [above] at (6.5,2.5) {$\partial^1_{2,1}$};
	\node [above] at (6.5,4) {$\partial^1_{2,2}$};
	
	\node [] at (2,1) {$\homology_0 C_{0,0}$};
	\node [] at (2,2.5) {$\homology_1 C_{0,1}$};
	\node [] at (2,4) {$\homology_2 C_{0,2}$};
	
	\node [] at (5,1) {$\homology_0 C_{1,0}$};
	\node [] at (5,2.5) {$\homology_1 C_{1,1}$};
	\node [] at (5,4) {$\homology_2 C_{1,2}$};
	
	\node [] at (8,1) {$\homology_0 C_{2,0}$};
	\node [] at (8,2.5) {$\homology_1 C_{2,1}$};
	\node [] at (8,4) {$\homology_2 C_{2,2}$};

	\node [right] at (8.5,0) {$p$};
	\node [above] at (0,4.5) {$q$};
\end{tikzpicture}
\end{center}
Computing homology yet again gives the second page~$\page^2_{p,q}=\homology_p \homology_q C_{p,q}$. There are also induced maps on the second page~$\partial^2_{p,q}\colon \page^2_{p,q}\rightarrow \page^2_{p-2,q+1}$. One can show that the following description on the level of representatives is well-defined. If~$[c]$ in~$\page^1_{p,q}=\homology_q C_{p,q}$ represents an element~$\alpha$ in~$\page^2_{p,q}$, then it is mapped to zero under~$\partial^1_{p,q}[c]=[\partial_h c]$. This in turn means that~$\partial_h c$ is in the image of~$\partial_v=\partial^0_{p-1,q+1}$. Hence there is a~$b$ in~$C_{p-1,q+1}$ such that~$\partial_v b= \partial_h c$ and applying~$\partial_h$ then produces an element~$\partial_h b$ which can be verified to represent an element in~$\page^1_{p-2,q+1}$. Denote by~$\beta$ the element in~$\page^2_{p-2,q+1}$ represented by~$[\partial_h b]$, and define~$\partial^2_{p,q}\alpha =\beta$. This is, of course, difficult to compute in general.

\begin{center}
\begin{tikzpicture}
	\draw (0,0) -- (8.5,0);
	\draw (0,0) -- (0,4.5);
	
	\draw[->] (7,1.25) -- (3,2.25);
	\draw[->] (7,2.75) -- (3,3.75);
	
	\node [above] at (6.25,1.5) {$\partial^2_{2,0}$};
	\node [above] at (6.25,3) {$\partial^2_{2,1}$};
	
	\node [] at (2,1) {$\homology_0 \homology_0 C_{0,0}$};
	\node [] at (2,2.5) {$\homology_0 \homology_1 C_{0,1}$};
	\node [] at (2,4) {$\homology_0 \homology_2 C_{0,2}$};
	
	\node [] at (5,1) {$\homology_1 \homology_0 C_{1,0}$};
	\node [] at (5,2.5) {$\homology_1 \homology_1 C_{1,1}$};
	\node [] at (5,4) {$\homology_1 \homology_2 C_{1,2}$};
	
	\node [] at (8,1) {$\homology_2 \homology_0 C_{2,0}$};
	\node [] at (8,2.5) {$\homology_2 \homology_1 C_{2,1}$};
	\node [] at (8,4) {$\homology_2 \homology_2 C_{2,2}$};

	\node [right] at (8.5,0) {$p$};
	\node [above] at (0,4.5) {$q$};
\end{tikzpicture}
\end{center}
The process now iterates:~$\page^3_{p,q}$ is defined as the homology at~$\page^2_{p,q}$. There are induced differentials~$\partial^3_{p,q}\colon \page^3_{p,q}\rightarrow \page^3_{p-3,q+2}$, much like in the case of~$\page^2$. What we end up with is a collection~$\page^r_{p,q}$ of vector spaces together with differentials~$\partial^r_{p,q}\colon \page^r_{p,q}\rightarrow \page^r_{p-r,q-1+r}$ satisfying that~$\page^{r+1}_{p,q}$ is obtained from~$\page^r_{p,q}$ by computing homology. Note that the process terminates; at some point~$\page^{r+n}_{p,q}\simeq \page^{r}_{p,q}$ for all~$n\geq 0$. This is because of how differentials must eventually be zero when they leave the first quadrant in the~$(p,q)$--plane. Let~$\page^{\infty}_{p,q}$ be the stable value of~$\page^r_{p,q}$. It is well-known that
\[
\homology_n \Tot C\simeq\oplus_{p+q=n} \page^{\infty}_{p,q}.
\]
Thus, if~$C=\F X$ for some bisimplicial set~$X$, then we have described a procedure to compute~$\homology_\ast \diag X$ from~$\Tot\F X$.

\subsection{The section spectral sequence}\label{section:sectionspectralsequence}
The previous Section implies the existence of a spectral sequence associated to~$\S_h$ which calculates the homology of~$\diag \S_h$.
\begin{definition}
The~\emph{section spectral sequence} of a height function~$h\colon X\rightarrow \R$ is the spectral sequence naturally associated to~$\S_h$.
\end{definition}
Entries on the zeroth page are determined by the free double complex~$\F \S_h$. Explicitly,~$\page^0_{p,q}=(\F\S_h)_{p,q}$ is the free vector space on~$\coprod_{\bar{a}\in \R_p} \S_h[\bar{a}]_q$, ranging over all non-decreasing real-valued sequences~$\bar{a}=(a_0,\dots,a_p)$. Differentials~$\partial^0_{p,q}\colon \page^0_{p,q}\rightarrow \page^0_{p,q-1}$ are induced from the alternating sum of vertical face maps~$\sum_i (-1)^i d_i^v$ in the spatial~$q$--direction. Computing homology vertically (in the~$q$--direction) thus produces the entries of the first page~$\page^1_{p,q}=\oplus_{\bar{a}\in\R_p} \homology_q \S_h[\bar{a}]$. Differentials on the first page are then induced in homology from the alternating sum of the horizontal face maps in the section~$p$--direction~$\sum_i (-1)^i\homology_qd_i^h~$. Proceeding as in Section~\ref{section:ss}, the section spectral sequence tells us how to recover the homology of~$\diag \S_h$, and thus of $X$.

\begin{proposition}[Proposition~\ref{prop:spectralsequenceintro}]
\label{prop:spectralsequence}
Let~$h\colon X\rightarrow \R$ be a height function. The associated section spectral sequence satisfies~$\page^2_{p,q}\simeq \homology_p \mathcal{G}_q$ and converges to the homology of~$X$:
\[
\homology_p \mathcal{G}_q \Rightarrow \homology_{p+q} X
\]
\end{proposition}
\begin{proof}
Theorem~\ref{intro:mainresult} tells us that~$\diag\S_h$ is homotopy equivalent, hence homology equivalent, to~$X$. So it only remains to verify that~$\page^2_{p,q}\simeq \homology_p\mathcal{G}_q$. The~$p$th entry of the~$q$th Reeb complex~$\mathcal{G}_q$ is 
\[
(\mathcal{G}_q)_p=\bigoplus \homology_p \S_h[\bar{a}]
\]
ranging over all increasing sequences in~$\R_p$, whereas~$\page^1_{p,q}$ ranges over all non-decreasing sequences. Hence we observe that~$\mathcal{G}_q$ is the non-degenerate complex~$\page^1_{-,q}/\mathrm{D}\page^1_{-,q}$. It is well-known that~$\page^1_{-,q}$ is chain homotopic to~$\page^1_{-,q}/\mathrm{D}\page^1_{-,q}$, see e.g.~\cite[p.150]{goerss2009simplicial}. In particular,~$\page^2_{p,q}=\homology_p \page^1_{-,q}$ is isomorphic to~$\homology_p\mathcal{G}_q=\homology_p\page^1_{-,q}/\mathrm{D}\page^1_{-,q}$.
\end{proof}

To summarize Proposition~\ref{prop:spectralsequence}: It does not matter if we exchange the~$q$th row~$\page^1_{-,q}$, including all non-decreasing sequences, with the Reeb complex~$\mathcal{G}_q$ (including only increasing sequences). Note the importance of this fact for making computations with section complexes feasible within finite time. 
%The reader unfamiliar with spectral sequences need not fixate too much on the above theory.
This is perhaps best illustrated through some simple examples:

\begin{example}
In Example~\ref{example:sh} we identified the first page of the section spectral sequence of~$X = \Delta^2\coprod_{\partial\Delta^2}\Delta^2$ with a single row -- the Reeb complex~$\mathcal{G}_0$. Hence, the differentials on the second page must be zero and we conclude that the homology of~$X$ and~$\mathcal{G}_0$ coincide. In Example~\ref{example:subdivision}, where the 2-simplices were subdivided prior to gluing, we are left with two non-zero rows on the first page:~$\mathcal{G}_0$ and~$\mathcal{G}_1$. Moreover, the only non-trivial entries of both~$\mathcal{G}_0$ and~$\mathcal{G}_1$ are in~$p=0,1$, implying that the differentials on the second page must be equal to zero. We calculate~$\homology_p\mathcal{G}_q$ and thus the second page in coordinates:
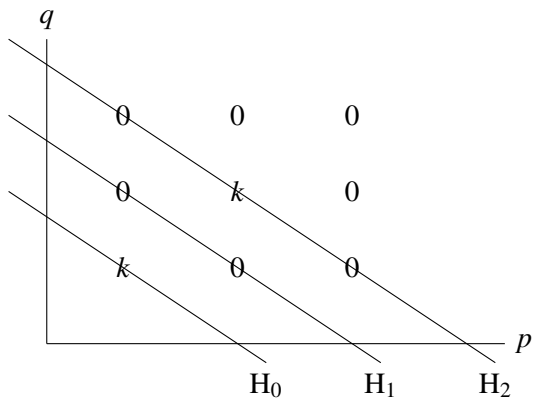
\begin{center}
\begin{tikzpicture}
	\draw (0,0) -- (6,0);
	\draw (0,0) -- (0,4);
	
	\node [] at (1,1) {$k$};
	\node [] at (1,2) {$0$};
	\node [] at (1,3) {$0$};
	\node [] at (2.5,1) {$0$};
	\node [] at (2.5,2) {$k$};
	\node [] at (2.5,3) {$0$};
	\node [] at (4,1) {$0$};
	\node [] at (4,2) {$0$};
	\node [] at (4,3) {$0$};

	\node [right] at (6,0) {$p$};
	\node [above] at (0,4) {$q$};
	
	\draw (-0.5, 4) -- (5.875, -0.25);
	\node [below] at (5.875, -0.25) {$\homology_2$};
	
	\draw (-0.5, 3) -- (4.375, -0.25);
	\node [below] at (4.375, -0.25) {$\homology_1$};
	
	\draw (-0.5, 2) -- (2.875, -0.25);
	\node [below] at (2.875, -0.25) {$\homology_0$};
\end{tikzpicture}
\end{center}
Since the sequence has converged, we extract~$\homology_0 X = k$,~$\homology_2 X=k$ and~$\homology_n X=0$ otherwise.
\end{example}

It was not a coincidence that the section spectral sequence from Example~\ref{example:subdivision} converged on the second page, as we shall make precise.
\begin{definition}
For a height function~$h\colon X\rightarrow \R$, we introduce the \emph{subdivision number} as the biggest~$n$ for which there is an increasing sequence~$\bar{a}=(a_0,\dots,a_n)$ such that the section space~$\S_h[\bar{a}]$ is non-empty.
\end{definition}

\begin{proposition}
Let~$h\colon X\rightarrow \R$ be a height function with subdivision number~$s$. The section spectral sequence collapses at the~$(s+1)$st page:~$\page^{n+s}_{p,q}\simeq \page^s_{p,q}$ for all~$n\geq 0$.
\end{proposition}
\begin{proof}
From the assumption, it follows that every Reeb complex~$\mathcal{G}_q$ (Example~\ref{example:explicitReebComplex}) has trivial entries above~$s$. Hence, the first page consists of zeros for~$p\geq s+1$ and the differentials on the~$(s+1)$st page must all terminate outside of the first quadrant; be equal to zero.
\end{proof}

Thus, the number of pages we need to compute is bounded by how subdivided~$X$ is relative to~$h\colon X\rightarrow \R$.

\begin{example}
\label{example:nontraivialSecondPage}
Construct a simplicial cylinder~$X$ by gluing together the leftmost and rightmost vertical~$1$--simplices in:
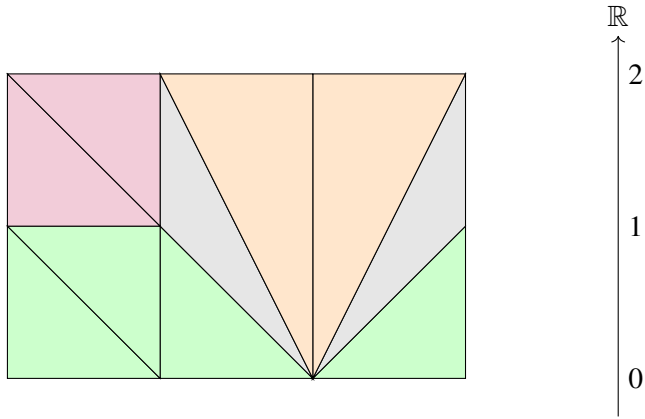
\begin{center}
\begin{tikzpicture}

    \draw[fill=green!20] (0,0) -- (2,0) -- (2,2) -- cycle;
    \draw[fill=gray!20] (0,0) -- (2,2) -- (2,4) -- cycle;
    \draw[fill=orange!20] (0,0) -- (2,4) -- (0,4) -- cycle;

    \draw[fill=green!20] (0,0) -- (-2,0) -- (-2,2) -- cycle;
    \draw[fill=gray!20] (0,0) -- (-2,2) -- (-2,4) -- cycle;
    \draw[fill=orange!20] (0,0) -- (-2,4) -- (0,4) -- cycle;
    
    \draw[fill=green!20] (-2,0) -- (-4,0) -- (-4,2) -- cycle;
    \draw[fill=green!20] (-2,0) -- (-4,2) -- (-2,2) -- cycle;
    \draw[fill=purple!20] (-2,2) -- (-4,2) -- (-4,4) -- cycle;
    \draw[fill=purple!20] (-2,2) -- (-4,4) -- (-2,4) -- cycle;
    
    \draw[->] (4,-0.5) -- (4,4.5);
    \node[right] at (4,0) {$0$};
    \node[right] at (4,2) {$1$};
    \node[right] at (4,4) {$2$};
    \node[above] at (4,4.5) {$\mathbb{R}$};

\end{tikzpicture}
\end{center}
A height function~$h\colon X\rightarrow \R$ is indicated by the right-hand values. Each section space of the form~$\S_h[a_0]$, equal to~$h^{-1}a_0$ for~$ a_0 = 0,1,2$, has one connected component. There are three~$1$--section spaces~$\S_h[a_0,a_1]$, all of which have a single connected component indicated by the simplices colored in green, orange and purple above. The section space~$\S_h[0,1,2]$ has two connected components represented by the gray~$2$--simplices. Only~$\S_h[0]$ and~$\S_h[2]$ have generators in~$\homology_1$, obtained by following the horizontal lines at the bottom and top of the cylinder. We mimic the calculations in Example~\ref{example:subdivision} to deduce
\[
\partial^1_{1,0} =\begin{bmatrix}
    -1 & -1 & 0\\
    1 & 0 & -1\\
    0 & 1 & 1
\end{bmatrix}\;\text{ and }\;
\partial^1_{2,0} =\begin{bmatrix}
    1 & 1 \\
    -1 & -1\\
    1 & 1 
\end{bmatrix}
\]
which gives the first page:
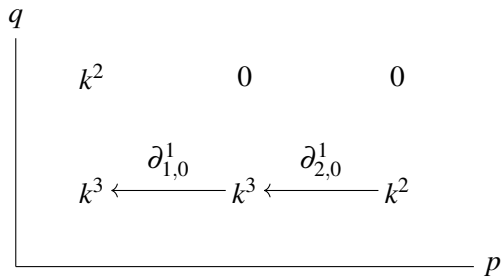
\begin{center}
\begin{tikzpicture}
	\draw (0,0) -- (6,0);
	\draw (0,0) -- (0,3);
	
	\node [] at (1,1) {$k^3$};
	\node [] at (3,1) {$k^3$};
	\node [] at (5,1) {$k^2$};
	\node [] at (1,2.5) {$k^2$};
	\node [] at (3,2.5) {$0$};
	\node [] at (5,2.5) {$0$};
	
	\draw[->] (2.75,1) -- (1.25,1); 
	\draw[->] (4.75,1) -- (3.25,1); 
	
	\node[above] at (2,1) {$\partial^1_{1,0}$};
	\node[above] at (4,1) {$\partial^1_{2,0}$};

	\node [right] at (6,0) {$p$};
	\node [above] at (0,3) {$q$};
	
\end{tikzpicture}
\end{center}
By computing homology again we obtain the second page:
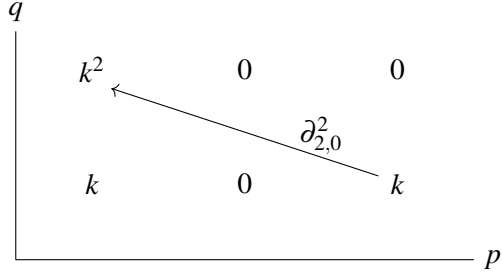
\begin{center}
\begin{tikzpicture}
	\draw (0,0) -- (6,0);
	\draw (0,0) -- (0,3);
	
	\node [] at (1,1) {$k$};
	\node [] at (3,1) {$0$};
	\node [] at (5,1) {$k$};
	\node [] at (1,2.5) {$k^2$};
	\node [] at (3,2.5) {$0$};
	\node [] at (5,2.5) {$0$};
	
	\draw[->] (4.75,1.1) -- (1.25,2.25); 
	
	\node[above] at (4,1.25) {$\partial^2_{2,0}$};

	\node [right] at (6,0) {$p$};
	\node [above] at (0,3) {$q$};
	
\end{tikzpicture}
\end{center}
The sequence must collapse on the next page, and, as the homology of~$X$ is not calculated yet, the differential cannot be zero. The representative in~$\page^2_{2,0}=k$ is given by the difference of the two gray~$2$--simplices. The alternating sum of the surrounding~$1$--simplices is in the image of~$\partial_v$. Geometrically, this happens by applying~$\partial_v$ to the sum of all~$2$--simplices (i.e.~$1$--simplices in the section spaces~$\S_h[a_0,a_1]$) not colored gray. Applying~$\partial_h$ to this sum gives the difference of the generators in~$\page^2_{0,1}\simeq k^2$. As an example, the top generator is obtained from the sum of the target of the purple and orange~$2$--simplices. We can thus conclude that~$\partial^2_{2,0}$ is the transpose of~$\begin{bmatrix}
1 & -1\\
\end{bmatrix}$. The third page only has two non-zero entries:~$\page^3_{0,0}\simeq k$ and~$\page^3_{0,1}\simeq k$. In particular, we calculate~$\homology_0 X=\homology_1 X= k$ and~$\homology_n X=0$ otherwise.
\end{example}

\section{Comparison to the continuous case}
\label{section:topologicalspaces!?}

In Section \ref{section:sectionsofheightfunction}, we saw that a height function $h \from X \to \R$ always associates to a piecewise linear function~$f\colon |X|\rightarrow \mathbb{R}$. Example \ref{example:standard2} illustrated that, in general, the topological space of sections $\mathrm{Sect}_f$ from \cite{trygsland2021} and the simplicial space of~$1$--sections $(\Sh)_1$ are significantly different.  A topological section of the form~$[a,b]\rightarrow T$ factorizes into smaller sections defined on~$[a,c]$ and~$[c,b]$ for any real number~$a\leq c\leq b$. Conversely, two sections~$\rho\colon [a,c]\rightarrow T$ and~$\tau \colon [c,b]\rightarrow T$ compose to a section on~$[a,b]$ via a canonical concatenation. This means that in contrast to the simplicial sections in~$(\Sh)_1$, the topological sections are automatically subdivided.
In particular, the spectral sequence obtained from~$\mathrm{Sect}_f$ terminates on the second page, reflecting the fact that all information about the homology of $T$ is contained in~$\mathrm{Sect}_f$. 
This is not true for~$(\Sh)_1$ in general, which led us to introduce higher sections. Consider now the case where the simplicial set $X$ is subdivided according to~$h$ (Definition  \ref{definition:subdivided}). Then we observed, in Section \ref{section:sectionspectralsequence}, that the spectral sequence associated to $\Sh$ terminates on the second page as well. Whenever $X$ is subdivided according to $h$, we can thus expect the space of~$1$--sections~$(\Sh)_1$ to contain the same homological information as the topological section space $\mathrm{Sect}_f$. The rest of this section is about making this observation into a formal statement which finally leads to a proof of Proposition \ref{proposition:comparereeb}. \\

 Generally, for fixed real values~$a \leq b$ we can define a map from the realization~$|\S_h[a,b]|$ to the space~$\mathrm{Sect}_f[a,b]$ as follows. A point in~$|\S_h[a,b]|$ is a class~$[\rho, \bar{t}]$ with~$\rho\colon \Delta^1\times \Delta^n\rightarrow X$ an~$n$--simplex in~$\S_h[a,b]$ and~$\bar{t}$ a point in the standard topological~$n$--simplex~$|\Delta^n|$. If we realize~$\rho$, then we obtain a continuous function~$|\rho|\colon |\Delta^1|\times |\Delta^n|\rightarrow |X|$ which hinges upon the existence of a homeomorphism~$|\Delta^1\times \Delta^n|\simeq |\Delta^1|\times |\Delta^n|$. For a fixed~$\bar{t}$, the restriction of~$|\rho|$ to~$|\Delta^1|\times \bar{t}$ is a section of~$f$ up to the linear orientation-preserving homeomorphism~$L_{a,b}\colon [a,b]\rightarrow |\Delta^1|$. Indeed, the composition~$h\circ \rho$ maps the unique non-degenerate~$1$--simplex in~$\Delta^1$ to~$a\leq b$ in~$\R$ regardless of its second component. See Definition~\ref{definition:ksections}. It follows that~$|h|\circ |\rho||_{|\Delta^1|\times \bar{t}}$ identifies~$|\Delta^1|$ with the~$1$--cell labeled by~$a\leq b$ in~$|X|$. Whence we define a continuous function~$\Phi_h\colon |\S_h[a,b]| \rightarrow \mathrm{Sect}_f[a,b]$ from the formula~$\Phi_h[\rho,\bar{t}]=|\rho|\circ (L_{a,b},\bar{t})$.

\begin{example}
\label{example:comparison}
Consider the height function~$h\colon \Delta^2\coprod_{\partial\Delta^2} \Delta^2 \rightarrow \R$ from Example~\ref{example:sh}. The section space~$\S_h[0,1]$ only consists of a single point represented by the~$1$--simplex~$0\rightarrow 1$, whereas the topological version~$\mathrm{Sect}_f[0,1]$ is a circle. Hence, the map~$\Phi_h$ cannot be a weak equivalence. While, if we subdivide~$\Delta^2\coprod_{\partial\Delta^2} \Delta^2$ as in Example~\ref{example:subdivision}, then~$\S_h[0,1]$ is a circle and~$\Phi_h$ is a weak equivalence.
\end{example}
\begin{proposition}
\label{proposition:comparisonWithTopological}
Assume that~$X$ is subdivided according to~$h\colon X\rightarrow \R$. For every pair of successive height levels~$a\leq b$, the continuous function~$\Phi_h\colon |\S_h[a,b]| \rightarrow \mathrm{Sect}_f[a,b]$ is a homology equivalence.
\end{proposition}
We can assume without loss of generality that the only non-empty height levels of~$h$ are~$0$ and~$1$.  
The strategy for the proof is then to shift all homological information of~$X$ into the space~$\S_h[0,1]$. This can be done by filling out all the simplices in the fibers of $h$ by by means of the following pushout 
\begin{equation*}
    \begin{tikzcd}
        h^{-1}(0) \coprod h^{-1}(1) \arrow[r,"(0\text{,}\id)"] \arrow[d,hookrightarrow] & \coprod\limits_{a=0,1}\sfrac{\Delta^1 \times h^{-1}(a)}{(1\text{,}x) \sim (1\text{,}y)} \arrow[ddr,bend left=20,"a"]  \arrow[d] \\
        X \arrow[drr,bend right=20,"h"] \arrow[r] & \tilde{X} \arrow[dr,"\tilde{h}"] \\ 
        & & \R
    \end{tikzcd}
\end{equation*}
from that we in particular get an induced height function~$\tilde{h} \from \tilde{X} \to \R$. 
\begin{lemma}\label{lemma:contractfibers}
Let~$h\from X \to \R$ be a height function that only meets~$a=0$ and~$b=1$ and let~$\tilde{h}\from \tilde{X}\to \R$ be the replacement constructed above. Then 
\begin{equation*}
    \S_h[0,1] = \S_{\tilde{h}}[0,1].
\end{equation*}

\end{lemma}
\begin{proof}
Postcomposing a section~$\rho \in \S_h[0,1]$ with the inclusion
\begin{equation*}
\begin{tikzcd}
    \Delta^1 \times \Delta^n \arrow[rr,"\rho"] \arrow[rd,"\tilde{\rho}"'] & & \tilde{X} \\
    & X \arrow[ru,hookrightarrow] &
    \end{tikzcd}
\end{equation*}
yields a section~$\tilde{\rho} \in \S_{\tilde{h}}$. Moreover, starting with any~$\tilde{\rho} \in \S_{\tilde{h}}$ it always factorizes like that. Indeed, if we assume that for a given section~$\rho \in \S_{\tilde{h}}[0,1]$ such a factorization does not exist. Then the image of this~$\rho$ contains a simplex that is not in~$X$. This simplex must then lie either in~$f^{-1}(0)$ or in~$f^{-1}(1)$ and thus be a horizontal face of~$\rho$. Furthermore it contain one of the two vertices in~$\tilde{X}$ which are not in~$X$. But as this vertex is clearly no horizontal face of any section we get a contradiction. Thus every section~$\rho \in \S_{\tilde{h}}[0,1]$ factors through~$X$ giving us the desired isomorphism.
\end{proof}
We can now proof Proposition \ref{proposition:comparisonWithTopological} by reducing to the case of contractible fibers. 

\begin{proof}[Proof of Proposition \ref{proposition:comparisonWithTopological}]
Let~$h \from X \to \R$ be a height function that, without loss of generality, only meets the height levels~$0$ and~$1$ and for which the fibers are contractible. Denote by~$\mathrm{T}\S_h$ the levelwise realization of the section complex like in Remark \ref{remark:bisimplicialsets}. We can then extend~$\Phi_h$ to a morphism of simplicial spaces 
\begin{equation*}
    \Phi_h \from |T\S_h| \to \nerve \mathrm{Sect}_f
\end{equation*}
that acts as the identity on zero-simplices. It follows from standard theory that the realization $|T\S_h|$ is isomorphic to~$|\diag \S_h|$. We combine this fact with the homotopy equivalences from Theorem \ref{intro:mainresult} of this paper and from Theorem 1.1 of \cite{trygsland2021}. This yields a commutative diagram 
\begin{equation*}
    \begin{tikzcd}
        \lvert T\S_h \rvert  \arrow[r,"\cong"]  \arrow[d,"\lvert \Phi_h \rvert"] & \lvert \diag \S_h \rvert \arrow[d,"\sim"]\\ 
        \lvert \nerve \mathrm{Sect}_f \rvert \arrow[r,"\sim"] & \lvert X \rvert
    \end{tikzcd}
\end{equation*}
  and exhibits $|\Phi_h|$ to be a homotopy equivalence as well.
 Using the result by Dold and Puppe \cite{dold1961homologie,goerss2009simplicial}, that was already mentioned in Section \ref{section:ss} gives us the commutative square
 \begin{equation*}
     \begin{tikzcd}
     \homology_{\ast}\Tot \F \S_h \arrow[r] \arrow[d] & \homology_{\ast}|T\S_h| \arrow[d,"\homology_{\ast}|\Phi_h|"] \\
     \homology_{\ast}\Tot \nerve \mathrm{Sect}_f \arrow[r] & \homology_{\ast}|\nerve \mathrm{Sect}_f|
     \end{tikzcd}
 \end{equation*}
 where all the arrows are isomorphisms. 
Consider now the two spectral sequences associated to $S_h$ and $\nerve \mathrm{Sect}_f$ respectively.
Because $X$ is subdivided according to $h$ these both converge on the second page. Combine this with the contractability of the fibers of $h$ to obtain for $q \geq 1$ the following extension of the above diagram
     \begin{equation*}
     \begin{tikzcd}
     \homology_q(\S_h)_1 \arrow[d,"\homology_q \Phi_h"] \arrow[r] & \homology_{q+1}\Tot \F \S_h  \arrow[d]  \\
     \homology_q \mathrm{Sect}_f \arrow[r] & \homology_{q+1}\Tot \nerve \mathrm{Sect}_f  & 
     \end{tikzcd}
 \end{equation*}
from which we can conclude that~$\homology_q \Phi_h$ is an isomorphism for all~$q\geq 1$. 
  
 For $q=0$ we have to do some extra work. This is because the horizontal differential \begin{equation*}
      \partial^1_{1,0} \from E^1_{1,0} \to E^1_{0,0}
  \end{equation*}
  is non-trivial in both spectral sequences. Its kernel-cokernel pair however induces the diagram 
  %Nevertheless we get the following commutative diagrams of isomorphisms: 
   % \begin{equation*}
    % \begin{tikzcd}
     %\homology_1 \homology_0 (\S_h)_1 \arrow[d,"\homology_1 \homology_0 \Phi_h"] \arrow[r] & \homology_{1}\Tot \F \S_h \arrow[r] \arrow[d] & \homology_{1}|T\S_h| \arrow[d,"\homology_{\ast}|\Phi_h|"] \\
     %\homology_1 \homology_0 \mathrm{Sect}_f  \arrow[r] & \homology_{1}\Tot \nerve \mathrm{Sect}_f \arrow[r] & \homology_{1}|\nerve \mathrm{Sect}_f|
%     \end{tikzcd}
 %\end{equation*}
  %and 
 %\begin{equation*}
  %   \begin{tikzcd}
   %  \homology_0 \homology_0 (h^{-1}(0)\sqcup h^{-1}(1))  \arrow[d,"\homology_1 \homology_0 \Phi_h"] \arrow[r] & \homology_{0}\Tot \F \S_h \arrow[r] \arrow[d] & \homology_{0}|T\S_h| \arrow[d,"\homology_{\ast}|\Phi_h|"] \\
    % \homology_0 \homology_0 (f^{-1}(0)\sqcup f^{-1}(1))  \arrow[r] & \homology_{0}\Tot \nerve \mathrm{Sect}_f \arrow[r] & \homology_{0}|\nerve \mathrm{Sect}_f|.
     %\end{tikzcd}
 %\end{equation*}
%These can be put together into the following commutative ladder 
\begin{equation*}
    \begin{tikzcd}[scale cd =0.8]
        0 \arrow[r] \arrow[d,"\cong"] & \homology_1 \homology_0 (\S_h)_1 \arrow[d,"\homology_1\homology_0 \Phi_h"] \arrow[r] & \homology_{0} (\S_h)_1 \arrow[d,"\homology_0\Phi_h"'] \arrow[r,"\partial^1_{1,0}"] & \homology_{0} (h^{-1}(0)\sqcup h^{-1}(1))  \arrow[d,"\cong"] \arrow[r] & \homology_0 \homology_0 (h^{-1}(0)\sqcup h^{-1}(1))  \arrow[d,"\homology_0 \homology_0 \Phi_h"] \\
        0 \arrow[r] & \homology_1 \homology_0 \mathrm{Sect}_f  \arrow[r] & \homology_{0} \mathrm{Sect}_f  \arrow[r,"\partial^1_{1,0}"] & \homology_{0} (f^{-1}(0)\sqcup f^{-1}(1))  \arrow[r] & \homology_0 \homology_0 (f^{-1}(0)\sqcup f^{-1}(1)) 
    \end{tikzcd}
\end{equation*}
where the horizontal rows are exact. Using a similar argument as for $\homology_q \Phi_h$ above we see that $\homology_1\homology_0 \Phi_h$ and $\homology_0\homology_0 \Phi_h$ are isomorphisms. An application of the five lemma exhibits $\homology_0\Phi_h$ as an isomorphism as well and thus concludes the proof for the case of contractible fibers. The more general case follows now with Lemma \ref{lemma:contractfibers} 
\end{proof}

Recall now the truncated Reeb complex~$\truncreeb_q^f$ from \cite{Vaupel_TDA_2022} that was already mentioned in the introduction. As an immediate consequence of Proposition~\ref{proposition:comparisonWithTopological} we finally get: 
\begin{corollary}\label{corollary:reebcomplexes}
Let~$h \from X \to \R$ be a simplicial height function, such that~$X$ is subdivided according to~$h$. Then~$h$ associates to a piecewise linear function~$f \from |X| \to \mathbb{R}$ and there is an induced isomorphism of chain complexes between~$\reeb_h$ and the continuous truncated Reeb complex~$\truncreeb_q$.
\end{corollary}

\section{Some final remarks and possible future directions}

%\subparagraph{Higher section spaces} were an important ingredient to our theory. Example \ref{example:standard2} showcased, that they are needed to incorporate simplices that span three or more height-levels into the section complex. In Section \ref{section:topologicalspaces!?} we discovered, that we can make contact with the world of continuous sections if we restrict ourselves to simplicial sets that are subdivided (see Defintion \ref{definition:subdivided}). In this case there are in particular no higher sections. At a glance the higher section spaces might thus seems as a technicality- barely necessary to make everything function as it should in the general case.      
%We believe however that they might provide a powerful modelling language for phenomena where  

We close this paper with two final remarks. The first concerns the relation of the section complex with the \textit{flow category} as defined in \cite{nanda2018discrete}. The flow category incorporates flow paths of a discrete Morse function on a simplicial complex. In a certain sense, it can be understood as a discrete variant of the construction from \cite{cohen1995morse}. We expect that there is an intimate relation with the theory developed in this paper and believe it could be a fruitful endeavour to make this relation precise.  \\

The second remark concerns our choice of modelling language in this paper - the simplicial sets.  
It is a common theme in applied topology to develop the theory in the category of topological spaces and then, in a second stage, conceive implementable algorithms. While these algorithms are inspired by the theory, their construction is often non-trivial and comes with its own set of complications. If we followed this paradigm, the present paper could have had a very different form. E.g., we could have built on the theory of \cite{trygsland2021} and discussed algorithms for discrete computations with topological Reeb complexes by approximating them with simplicial complexes. Instead, we decided to develop a theory of sections for simplicial sets. While this theory is certainly inspired by the continuous version, it can also stand on its own feet. This is possible because simplicial sets provide a model for spaces that is equal in power to the continuous one. There is a rich simplicial homotopy theory and often a clean categorical treatment in terms of universal constructions is made possible by the fact that simplicial sets are presheaves. If need be, one can translate to topological spaces in terms of the Quillen equivalence given by the realization functor. Furthermore, we note that the theory of simplicial sets contains that of simplicial complexes as a special case. \\
All this makes us believe that it can be beneficial to formulate theories of relevance to computational topology directly in simplicial sets. While still modelling the homotopy type of all spaces, such a theory is inherently combinatorial. This can make the subsequent development of implementable algorithms straightforward and streamlined with the theory. Furthermore, many of the intended applications may be of a discrete nature anyways. Using simplicial sets as the preferred modelling language might reveal important phenomena intrinsic to these discrete systems. An example that came up in the present paper are the \textit{higher section spaces}. These were important through the presence of higher differentials in the general section spectral sequence. Contrast this with the continuous theory, where the spectral sequence converges on the second page. It is conceivable that higher section spaces prove to be more then a technicality and provide a valuable tool in the modelling and analysis of discrete height functions.

\section*{Acknowledgements}
We would like to thank our supervisors Benjamin Dunn and Markus Szymik for their valuable input and the kind and motivating encouragement. This work was financially supported from the Department of Mathematical Sciences at the NTNU and from NTNU's Enabling Technologies Biotechnology program.

\newpage

\addcontentsline{toc}{section}{References}
\bibliographystyle{amsalpha}
\bibliography{SectionComplexes}

\end{document}